\RequirePackage{fix-cm}

\documentclass[smallextended,envcountsect]{svjour3}
\smartqed
\usepackage{graphicx}
 \usepackage{mathptmx}
\usepackage{latexsym}
\usepackage{geometry}
\usepackage{hyperref}
\usepackage{amsmath}
\usepackage{amsfonts,amssymb}
\usepackage{amssymb}
\usepackage{indentfirst}
\usepackage{mathrsfs}
\usepackage{appendix}
\usepackage{arydshln}
\usepackage{enumerate}
\usepackage{graphicx}
\usepackage{subfigure}
\usepackage{multirow}
\usepackage{booktabs}
\usepackage{xcolor}
\usepackage{latexsym}
\usepackage{graphicx}
\usepackage{epstopdf}
\usepackage{bbm}
\usepackage{bm}
\usepackage{mathrsfs}
\usepackage{graphicx}
\usepackage[misc]{ifsym}

\begin{document}

\title{A Fast Fourier-Galerkin Method for Solving Boundary Integral Equations on Non-axisymmetric Toroidal Surfaces}



\author{Yiying Fang  \textsuperscript{1}       \and
       Ying Jiang \textsuperscript{2,~\Letter} \and
       Jiafeng Su \textsuperscript{3}}

\institute{
\begin{itemize}
    \item[\textsuperscript{\Letter}] {Ying Jiang}\\
    \email{jiangy32@mail.sysu.edu.cn}
    \at
    \item[\textsuperscript{1}]
        School of Computer Science and Engineering, Sun Yat-sen University, Guangzhou 510006, People's Republic of China.
        \email{fangyy8@mail3.sysu.edu.cn}
                        %  \\
    \item[\textsuperscript{2}]
    School of Computer Science and Engineering, Sun Yat-sen University, Guangzhou 510006, People's Republic of China, and Guangdong Province Key Laboratory of Computational Science, Guangzhou 510275, People's Republic of China.
    \email{jiangy32@mail.sysu.edu.cn}
    \item[\textsuperscript{3}]
    School of Computer Science and Engineering, Sun Yat-sen University, Guangzhou 510006, People's Republic of China.
    \email{sujf7@mail2.sysu.edu.cn}
\end{itemize}
}

\maketitle

\begin{abstract}
In this paper, we introduce a fast Fourier-Galerkin method for solving boundary integral equations on smooth non-axisymmetric toroidal surfaces. We analyze the properties of the integral operator's kernel to derive the decay pattern of the entries in the representation matrix. Leveraging this decay pattern, we devise a truncation strategy that efficiently compresses the dense representation matrix of the integral operator into a sparser form containing only $\mathcal{O}(N\ln^2 N)$ nonzero entries, where $N$ denotes the degrees of freedom of the discretization method. We prove that this truncation strategy achieves a quasi-optimal convergence order of $\mathcal{O}(N^{-p/2}\ln N)$, with $p$ representing the degree of regularity of the exact solution to the boundary integral equation. Additionally, we confirm that the truncation strategy preserves stability throughout the solution process. Numerical experiments validate our theoretical findings and demonstrate the effectiveness of the proposed method.
\end{abstract}

\section{Introduction}
\label{intro}
\noindent
Boundary integral equation (BIE) methods have greatly advanced research in solving both initial and boundary value partial differential equations \cite{atkinson1996numerical,hsiao2008boundary,kress2013linear}. These methods are known for: (1) eliminating volume mesh generation; (2) producing operators with bounded condition numbers; (3) precisely meeting far-field boundary conditions in exterior problems; and (4) achieving high convergence rates with smooth domain boundaries and conditions. Due to these benefits, BIE methods are applied in various fields including analyzing the scattering of acoustic and electromagnetic waves across different media or obstacles \cite{Bao2019,Colton2013Integ,jun2019fft}, solving multi-medium elasticity \cite{Yang2015}, Stokes flows \cite{Bystricky2021}, and both steady and transient heat conduction problems \cite{jiang2020}. Furthermore, with suitable time discretization or linearization, BIE methods address complex nonlinear partial differential equations, such as fluid-structure interactions \cite{Opstal2015}, dynamic poroelasticity \cite{Zhang2021}, the Navier-Stokes equations \cite{Klinteberg2020}, and multiphase flows \cite{Wei2020}.

With advances in high-order accurate techniques \cite{cai2018fast,jiang2014fast,Kapur1997,Kolm2001}, achieving near double-precision accuracy in 2D with a moderate number of degrees of freedom has become feasible, significantly reducing computational complexity to linear or quasi-linear levels. Research continues for solving BIEs on 3D surface, notably with \cite{atkinson2012}'s development of a Spectral Galerkin method using spherical harmonics for equations on surfaces topologically equivalent to a sphere.
This method achieves convergence rates of $\mathcal{O}(N^{-\frac{p+\alpha-1/2}{2}})$, with degrees of freedom scaling as $\mathcal{O}(N)$, when the solution belongs to the weighted continuously differentiable function space $C^{p,\alpha}$ ($p\geq 0$, $\alpha\in (0,1]$ and $p+\alpha>1/2$).
The work presented in  \cite{Ying2006} utilizes Nystr\"{o}m’s method to address 3D elliptic boundary value problems on domains with smooth boundaries. It achieves an error bound $\mathcal{O}(N^{-\frac{p-1}{4}})$, where $N$ represents the number of discretization points on the boundary of the domain, and the boundary data is in $C^p$ .

Recently, numerous effective algorithms have emerged for solving BIEs on rotationally symmetric surfaces. These equations can typically be expressed as a sequence of BIEs defined on a generating curve. A high-order efficient scheme utilizing Nystr\"{o}m discretization and quadrature techniques has been developed for these equations on axisymmetric surfaces \cite{Young2012nystrom}. Specific algorithms for solving Helmholtz equations \cite{Serkh2022effic,Helsing2014explicit} and Maxwell's equations \cite{Charles2019high,jun2019fft} have significantly simplified the solution process. The computational cost for forming the corresponding linear systems of BIEs scales as $\mathcal{O}(N^{3/2}\log N)$, where $N$ represents the number of discretization points on the boundary of the domain \cite{Young2012nystrom}.

Over the past few decades, significant advancements have been made in developing fast solution algorithms for boundary integral equations. The integration of the Fast Multipole Method (FMM) with traditional discretization techniques like the Nyström and Galerkin methods has significantly enhanced the efficiency of solving two-dimensional and three-dimensional boundary integral equations \cite{1999Cheng,1997Greengard,2018Klockner}. Notable advances in higher-order numerical integration and the handling of singular integrals have led to substantial accuracy improvements. Innovations such as the hybrid Gauss-Trapezoidal quadrature rule \cite{1999Alpert} and refined techniques for evaluating layer potentials near boundaries \cite{2014Barnett,2008Helsing} have been pivotal. Additionally, the integration of the expansion-based quadrature method (QBX) with FMM technology has optimized calculations for layer potentials close to boundaries, significantly reducing computational complexity \cite{2013Kloeckner,2017Rachh}. The relationship between the number of multipole expansion terms, $q$, and the computational complexity of the fast multipole method has been precisely characterized. For instance, the computational cost of the $N \log N$ scheme in three dimensions is approximately $(27+189q^2\log_8 N)N$ (see (5.19) in \cite{Beatson1997}). By leveraging the fast multipole method, researchers have significantly reduced the computational complexity of two- and three-dimensional boundary integral equations, thereby expanding the applicability of these methods in physics and engineering.

In addition, boundary integral equation solvers defined on smooth toroidal surfaces have gained significant attention due to their effectiveness in efficiently computing Taylor states within toroidal geometries \cite{Malhotra2019,neil2018integr}. The computation of Taylor states in these geometries is crucial for the calculation of stepped pressure stellarator equilibria, and plays an important role in both the design of new non-axisymmetric magnetic confinement devices as well as the analysis of experimental results from existing ones \cite{neil2018integr}.
However, it is not straightforward for extending the algorithms \cite{Charles2019high,Serkh2022effic,Helsing2014explicit,jun2019fft,Young2012nystrom} developed for axisymmetric surfaces to boundary integral equations on non-axisymmetric surfaces. This is because, in axisymmetric geometries, the boundary of the domain can be regarded as a closed curve revolving around a particular axis, allowing for a simplification of the boundary integral equations. In non-axisymmetric geometries, such simplifications are generally not feasible. This paper aims to bridge this gap by considering a more general case and developing a fast Fourier-Galerkin algorithm to solve the boundary integral equation derived from the interior Dirichlet problem on surfaces that, while not necessarily axisymmetric, are diffeomorphic to a torus.

%In fact, the existing literature on efficient and accurate methods for solving boundary integral equations on non-axisymmetric surfaces is relatively sparse, and there is a lack of high-precision algorithms with linear or quasi-linear complexity (referring here to the number of non-zero elements in the discretized coefficient matrix).

Fast Fourier-Galerkin methods have been widely used to solve two-dimensional BIEs on smooth boundaries \cite{cai2018fast,jiang2014fast,2021jiangfast}. The key concept behind these methods is to design truncation strategies for the representation matrices of integral operators with continuous or weakly singular kernels under the Fourier bases. By truncating a dense matrix to a sparse one while retaining essential entries, a high order of accuracy can be achieved. It is worth noting that the kernel functions employed in these methods for solving two-dimensional BIEs are typically Lebesgue square integrable functions. However, in the case considered in this paper, the kernel functions are not necessarily Lebesgue square integrable but are Lebesgue integrable, which poses additional challenges. Fortunately, after performing a shear transformation on the kernel function and fixing the sheared variables, we discovered a uniform analytical continuation radius for each of the remaining variables. Building on this observation, we can deduce the decay pattern of the entries in the matrix representing the three-dimensional boundary integral operator, enabling the design of an effective truncation strategy.

Consequently, using this truncation strategy, we propose a \emph{fast} Fourier-Galerkin method to efficiently solve boundary integral equations on non-axisymmetric toroidal surfaces. This method achieves a quasi-optimal convergence order of $\mathcal{O}(N^{-p/2}\log N)$ with a quasi-linear number of non-zero entries in the representation matrix, scaling as $\mathcal{O}(N\log^2 N)$.
Here, $p$ represents the degree of regularity of the exact solution, and $N$ denotes the degrees of freedom in this Fourier-Galerkin method. In general, for the Sobolev space
$H_p$, composed of bivariate functions with $p$-th order regularity, the projection error of these functions typically scales as $\mathcal{O}(N^{-p/2})$. It should be noted that in this article, we focus solely on the truncation strategy and do not address the numerical integration required for computing the matrix elements, which will be the subject of our future work.
	
The paper is organized as follows: In Section \ref{sec:2}, we introduce the fast Fourier-Galerkin method along with a truncation strategy for compressing the representation matrix of the integral operator. In Section \ref{sec:3}, we demonstrate that after a shear transformation of the kernel function and fixation of the sheared variables, a uniform analytical continuation radius exists for the remaining two variables. This enables us to deduce the decay pattern of the coefficients in the representation matrix. In Section \ref{sec:4}, we analyze the stability and convergence order of the fast Fourier-Galerkin method. Section \ref{sec:5} demonstrates the application of our method to solve boundary integral equations on non-axisymmetric surfaces through numerical experiments, confirming the method's high accuracy and efficiency. Finally, Section \ref{sec:6} concludes the paper with a discussion of the findings and potential future research directions.

\section{A fast Fourier–Galerkin method}
\label{sec:2}
In this section, we introduce a fast Fourier-Galerkin method for solving BIEs that derived from the interior Dirichlet problem on a multiply connected bounded open region $\Omega\subset \mathbb{R}^3$ with a smooth boundary $\partial \Omega$. Specifically, we consider cases where $\partial \Omega$ is a $C^\infty$ toroidal surface which is not necessarily axisymmetric. This mapping implies an infinitely differentiable bijection from the torus to $\partial \Omega$, which also possesses a differentiable inverse.

The interior Dirichlet problem involves finding a function $u$ satisfying
\begin{equation*}\label{Diric}
	\begin{aligned}
		& \Delta u(\mathbf{x})=0,\quad && \mathbf{x}\in \Omega,\\
		& u(\mathbf{x})=h(\mathbf{x}),\quad &&\mathbf{x}\in\partial \Omega,
	\end{aligned}
\end{equation*}
where $h$ is a given sufficiently smooth function on $\partial \Omega$.
For a vector $\mathbf{a}=[a_k:k\in\mathbb{Z}_3]\in\mathbb{C}^3$, its 2-norm is defined as $\|\mathbf{a}\|_2:=\sqrt{\sum_{j\in\mathbb{Z}_3}|a_j|^2}$.
It is well known that the solution $u$ can be represented through a double layer potential as established in Theorem 6.22 of \cite{kress2013linear}
\begin{equation*}
	u(\mathbf{x})=\int_{\partial \Omega}\varrho(\mathbf{y})\frac{\partial}{\partial \boldsymbol{\nu}_{\mathbf{y}}}\left(\frac{1}{\|\mathbf{x}-\mathbf{y}\|_2}\right){d}s(\mathbf{y}),\quad \mathbf{x}\in \Omega,
\end{equation*}
where $\frac{\partial}{\partial \boldsymbol{\nu}_{\mathbf{y}}}$ denotes the normal derivative in the direction of $\boldsymbol{\nu}_{\mathbf{y}}$, the outward unit normal vector at $\mathbf{y}\in\partial \Omega$, and ${d}s(\mathbf{y})$ represents the surface element at  $\mathbf{y}\in\partial \Omega$. Here, $\varrho$ represents an unknown double layer density function determined by solving the boundary integral equation
\begin{equation}\label{bie1}
	\varrho(\mathbf{x})-\frac{1}{2\pi}\int_{\partial \Omega}\varrho(\mathbf{y})\frac{\partial}{\partial \boldsymbol{\nu}_{\mathbf{y}}}\left(\frac{1}{\|\mathbf{x}-\mathbf{y}\|_2}\right)ds(\mathbf{y})=-\frac{1}{2\pi}h(\mathbf{x}),\quad \mathbf{x}\in\partial \Omega.
\end{equation}
This equation is known to have a weakly singular kernel and is uniquely solvable, as detailed in section 9.1.4 of \cite{atkinson1996numerical} and by Theorem 6.23 in \cite{kress2013linear}.
	
To propose the \emph{fast} Fourier-Galerkin method for solving \eqref{bie1}, we require a parametrization for the boundary $\partial \Omega$.
Define $\mathbb{N}$ as the set of positive integers. For any given set $\mathbb{A}$ and $j\in\mathbb{N}$, the $j$-fold tensor product of $\mathbb{A}$ is denoted by $\mathbb{A}^j:=\mathbb{A}\otimes\mathbb{A}\otimes\cdots\otimes\mathbb{A}$.
For a vector $\mathbf{v}\in\mathbb{C}^n$, with $n \in \mathbb{N}$, its transpose is denoted by $\mathbf{v}^T$.
Suppose $\partial \Omega$ is characterized by the differentiable parametric equation $\Gamma(\boldsymbol{\theta}):=\left[\gamma_0(\boldsymbol{\theta}),\gamma_1(\boldsymbol{\theta}),\gamma_2(\boldsymbol{\theta})\right] {^T}$, $\boldsymbol{\theta}\in I^2_{2\pi}$, where $I_{2\pi}:=[0,2\pi)$, and $\gamma_j$ is $2\pi$-biperiodic on $\mathbb{R}^2$. Additionally, let $L_2(I_{2\pi}^2)$ denote the standard Hilbert space of square-integrable functions over $I_{2\pi}^2$ with norm $\|\cdot\|$. For $n\in\mathbb{N}$, define $\mathbb{Z}_{n}:=\{0,1,\cdots,n-1\}$.
For all $\boldsymbol{\theta}:=[\theta_0,\theta_1]\in I_{2\pi}^2$, let $d\boldsymbol{\theta}:=d\theta_0d\theta_1$ and $\frac{\partial\Gamma}{\partial \theta_{\iota}}:=\left[\frac{\partial\gamma_0}{\partial\theta_{\iota}}, \frac{\partial\gamma_1}{\partial\theta_{\iota}}, \frac{\partial\gamma_2}{\partial\theta_{\iota}}\right]^T$, where $\iota\in\mathbb{Z}_2$.
For vectors $\mathbf{a}:=[a_k:k\in\mathbb{Z}_3]^T$ and $\mathbf{b}:=[b_k:k\in\mathbb{Z}_3]^T$ in $\mathbb{C}^3$, denote their dot product as $\mathbf{a}\cdot\mathbf{b}:=\sum_{j\in\mathbb{Z}_3}a_jb_j$, and their cross product as  $\mathbf{a}\times\mathbf{b}:=[a_1b_2-a_2b_1,a_2b_0-a_0b_2,a_0b_1-a_1b_0]^T$.
By substituting the parametric equation $\Gamma$ into \eqref{bie1}, we reformulate the boundary integral equation as follows,
\begin{equation}\label{orig_bie}
	(\mathcal{I}-\mathcal{K})\rho=g,
\end{equation}
where $g(\boldsymbol{\theta}):=-\frac{1}{2\pi}h(\Gamma(\boldsymbol{\theta}))$, $\mathcal{I}$ is the identity operator, and $\mathcal{K}$ is the integral operator defined by
\begin{equation*}
	(\mathcal{K}\varphi)(\boldsymbol{\theta}):=\int_{I^2_{2\pi}} K(\boldsymbol{\theta}, \boldsymbol{\eta})\varphi( \boldsymbol{\eta})d \boldsymbol{\eta},\quad {\rm for~} \varphi\in L_2(I_{2\pi}^2) {~\rm and~} \boldsymbol{\theta}\in I_{2\pi}^2.
\end{equation*}
Here, the kernel function $K$ is specified as
\begin{equation}\label{kernel_K_specific}
	K(\boldsymbol{\theta}, \boldsymbol{\eta}):=-\frac{1}{2\pi}\frac{(\Gamma( \boldsymbol{\theta})-\Gamma( \boldsymbol{\eta}))\cdot \left(\frac{\partial\Gamma}{\partial \eta_0}\times\frac{\partial\Gamma}{\partial \eta_1}\right)(\boldsymbol{\eta})}{\| \Gamma( \boldsymbol{\theta})-\Gamma( \boldsymbol{\eta}) \|_2^3}.
\end{equation}

To guarantee the existence and uniqueness of the solution for \eqref{orig_bie}, the following assumption about the parametric equation $\Gamma$ is needed

\noindent \textbf{(A1)} there exists a constant $C_0>0$ such that for all $\boldsymbol{\theta}, \boldsymbol{\eta}\in I_{2\pi}^2$,
$$
\|\Gamma(\boldsymbol{\theta})-\Gamma(\boldsymbol{\eta})\|_2\geq C_0  \zeta(\boldsymbol{\theta}-\boldsymbol{\eta}),
$$
where $\zeta(\boldsymbol{\theta}):=\sqrt{(\min\{\theta_0, 2\pi-\theta_0\})^2+(\min\{\theta_1, 2\pi-\theta_1\})^2}$.
	
Assumption \textbf{(A1)} implies that $\Gamma$ is a bijective map from $\partial D$ to $I_{2\pi}^2$. Supported by the proof of Theorem 6.23 in \cite{kress2013linear}, this assumption enables us to establish the existence and uniqueness of the solution for \eqref{orig_bie}.

We revisit the classical Fourier-Galerkin method used to solve the equation \eqref{orig_bie}. We define $\mathbb{Z}$ as the set of all integers. For all $m,n\in\mathbb{Z}$ with $m\leq n$, let $\mathbb{Z}_{m,n}:=\{k\in\mathbb{Z}:m\leq k\leq n\}$. For each $k\in\mathbb{Z}$, we define $e_k(\phi):=\frac{1}{\sqrt{2\pi}}e^{ik\phi}$ for $\phi\in I_{2\pi}$, where $i$ represents the imaginary unit. Additionally, for any pair $\mathbf{k}:=[k_0,k_1]\in\mathbb{Z}^2$, we set $e_{\mathbf{k}}(\boldsymbol{\theta}):=e_{k_0}(\theta_0)e_{k_1}(\theta_1)$ for $\boldsymbol{\theta}\in I_{2\pi}^2$. For all $n\in\mathbb{Z}$, let $X_n$ denote the finite dimensional space spanned by the basis $\{e_{\mathbf{k}}:\mathbf{k}\in\mathbb{Z}_{-n,n}^2\}$, and define $\mathcal{P}_n$ as the orthogonal projection from ${L_2}(I_{2\pi}^2)$ onto $X_n$. Then $\mathcal{P}_n$ is given by $\mathcal{P}_n \varphi=\sum_{\mathbf{k}\in\mathbb{Z}_{-n,n}^2}\left<\varphi,e_\mathbf{k}\right>e_\mathbf{k}$ for all $\varphi\in {L_2}(I_{2\pi}^2)$, where $\left<\cdot,\cdot\right>$ denotes the inner product in ${L_2}(I_{2\pi}^2)$. The goal of the Fourier-Galerkin method for solving \eqref{orig_bie} is to find $\rho_n\in X_n$ such that	
\begin{equation}\label{Galerkin_operator_eq}
	\rho_n-\mathcal{K}_n\rho_n=\mathcal{P}_n g,
\end{equation}
where $\mathcal{K}_n:=\mathcal{P}_n\mathcal{K}$.
For convenience, we \emph{denote the degrees of freedom in \eqref{Galerkin_operator_eq} as $N=(2n+1)^2$ throughout this paper}. Define $\mathbf{g}_N:=\left[\left<g,e_{\mathbf{k}}\right>:\mathbf{k}\in\mathbb{Z}_{-n,n}^2\right]$, and $\mathbf{K}_N:=\left[K_{\mathbf{k},\mathbf{l}}:\mathbf{k},\mathbf{l}\in\mathbb{Z}_{-n,n}^2\right]$, where $K_{\mathbf{k},\mathbf{l}}:=\left<\mathcal{K}e_{\mathbf{l}}, e_{\mathbf{k}}\right>$. Therefore, solving equation \eqref{Galerkin_operator_eq} is equivalent to finding $\boldsymbol{\rho}_N$ that satisfies
\begin{equation}\label{Galerkin_linear_system}
	(\mathbf{I}_N-\mathbf{K}_N)\boldsymbol{\rho}_N=\mathbf{g}_N,
\end{equation}
where $\mathbf{I}_N$ is the identity matrix.

Typically, the entries of $\mathbf{K}_N$ are nonzero. Constructing this $N$-th order matrix requires the computation of $\mathcal{O}(N^2)$ quadruple integrals, leading to significantly higher computational costs as $n$ increases. To address this challenge, a truncation strategy that approximates $\mathbf{K}_N$  with a sparse matrix is essential.
Truncation strategies for the Fourier-Galerkin method have been extensively studied when solving two-dimensional BIEs \cite{cai2018fast,jiang2014fast,2021jiangfast}. In two-dimensional cases, the kernel functions of BIEs are usually square-integrable, which facilitates in analyzing the decay patterns of the coefficient matrix entries and designing truncation strategies that can achieve high accuracy and linear or quasi-linear complexity. However, in three-dimensional cases, the kernel functions are not square-integrable, and the decay patterns of their entries under discretization by Fourier basis are not well understood. This means that the truncation strategies used in two-dimensional cases cannot be directly applied to three-dimensional BIEs.

Fortunately, the matrix $\mathbf{K}_N$ is well-organized in the sense that the very few non-negligible entries occur near a few
shifted diagonals, while the entries decay rapidly in the direction of the anti-diagonal. These behaviors are illustrated in Figure \ref{fig:coef}, where image (a) shows the modulus values of the entries in the matrix $\mathbf{K}_{121}$. In image (b) of Figure \ref{fig:coef}, the red dashed line represents the modulus values of the entries along the main diagonal of $\mathbf{K}_{121}$, while the green solid line represents those along the anti-diagonal. Based on these observations, we propose a truncation strategy for $\mathbf{K}_N$ aimed at efficiently reducing computational demands while preserving accuracy. A rigorous analysis of the decay pattern of the entries in $\mathbf{K}_N$ will be provided in Section \ref{sec:3}. This analysis will help estimate the error introduced by our truncation strategy and confirm that the strategy preserves the convergence order of the traditional Fourier-Galerkin method.

%In Section 3, we will delve into the decay pattern of the entries in $\mathbf{K}_N$. This analysis will aid in estimating the error introduced by our truncation strategy and confirm that the strategy maintains the convergence order of the traditional Fourier-Galerkin method.

\begin{figure}[h]
\centering
\begin{minipage}[c]{5.2cm}
\includegraphics[height=4.4cm,width=5.2cm]{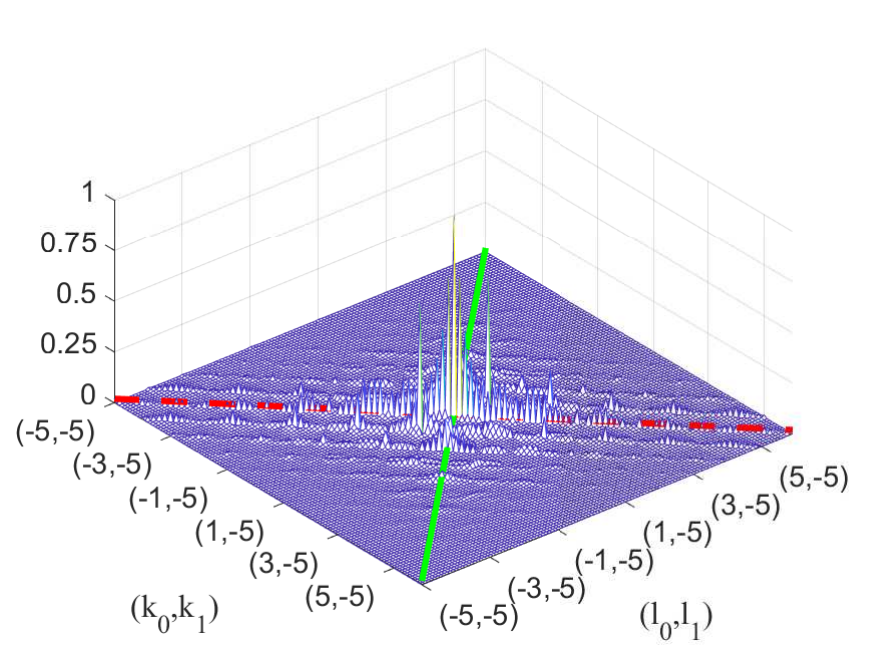} \centering (a)
\end{minipage}
\hspace{0.5cm}
\begin{minipage}[c]{6.0cm}
\includegraphics[height=4.4cm]{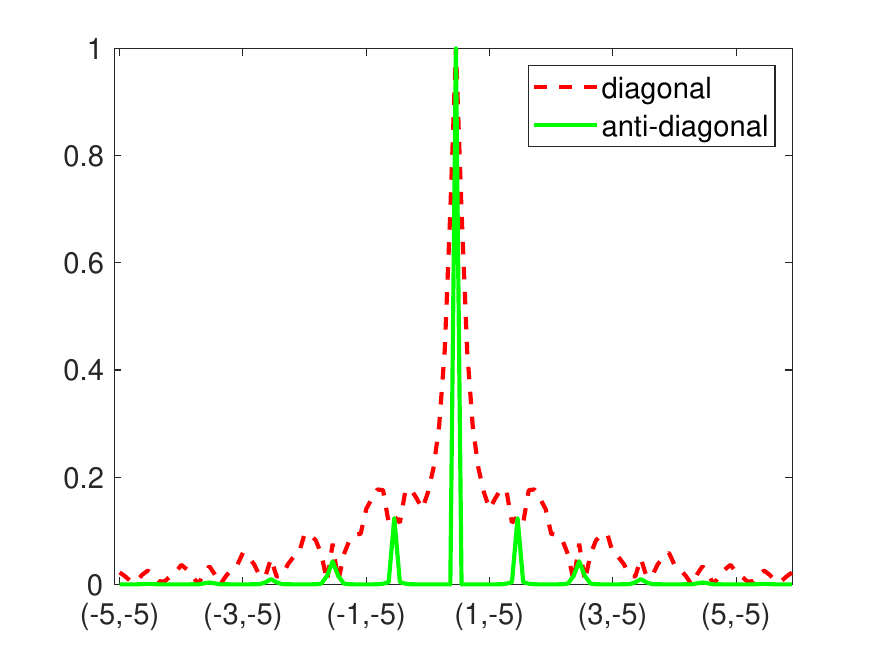}
\centering (b)
\end{minipage}\\
\begin{minipage}[c]{5.2cm}
\includegraphics[height=4.4cm,width=5.2cm]{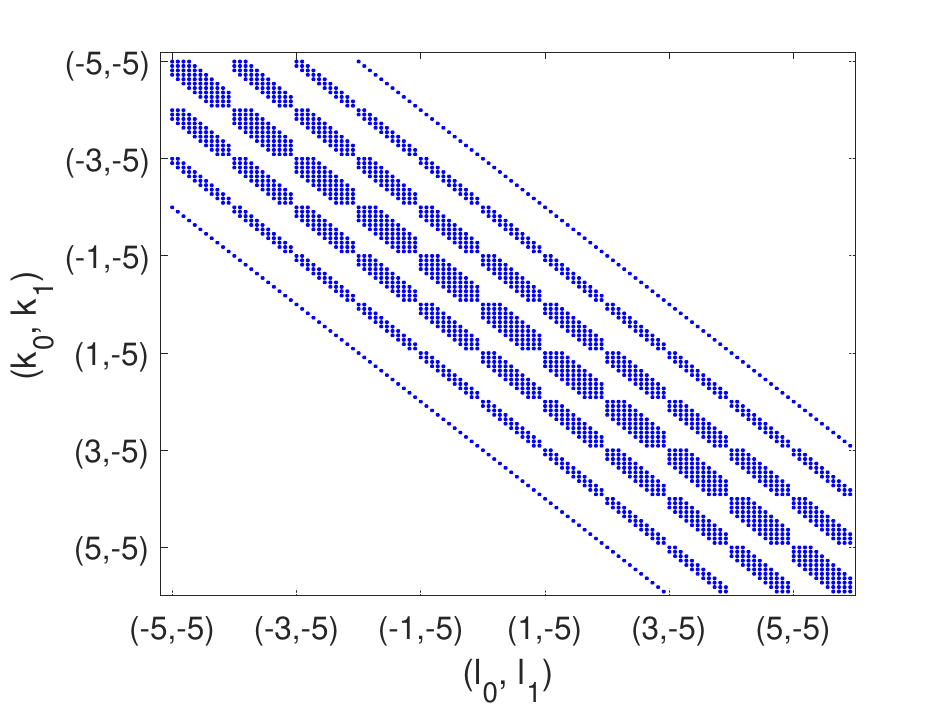}
\centering (c)
\end{minipage}
\caption{(a) Modulus values of the entries in $\mathbf{K}_{121}$; (b) Modulus values of the entries on the main diagonal and anti-diagonal of $\mathbf{K}_{121}$; (c) Distribution of nonzero entries of $\widetilde{\mathbf{K}}_{121}$.}
\label{fig:coef}
\end{figure}
	
We develop the truncation strategy as follows. For all $\mathbf{k}:=[k_0,k_1]\in\mathbb{Z}^2$,  we define $\|\mathbf{k}\|_1:=|k_0|+|k_1|$. For $n\in\mathbb{N}$ and $q>0$, let  $\mathbb{L}_N(q):=\left\{(\mathbf{k},\mathbf{l}) :\mathbf{k},\mathbf{l}\in \mathbb{Z}_{-n,n}^2 {\rm ~and~} \|\mathbf{k}-\mathbf{l}\|_1\leq q\ln N\right\}$. The truncation strategy involves retaining the matrix entries $K_{\mathbf{k}, \mathbf{l}}$ for $(\mathbf{k}, \mathbf{l})\in \mathbb{L}_N(q)$, and replacing all other entries with zeros. Specifically, for each pair $(\mathbf{k},\mathbf{l})\in \mathbb{Z}_{-n,n}^4$, we define
$$
\widetilde{K}_{\mathbf{k},\mathbf{l}}:=\left\{
\begin{aligned}
	&K_{\mathbf{k},\mathbf{l}},\quad &&(\mathbf{k},\mathbf{l})\in\mathbb{L}_N(q),\\
	&0,\quad &&\mathrm{otherwise}.
\end{aligned}
\right.
$$
This results in a sparse matrix  $\widetilde{\mathbf{K}}_N:=\left[\widetilde{K}_{\mathbf{k},\mathbf{l}}:\mathbf{k},\mathbf{l}\in\mathbb{Z}_{-n,n}^2\right]$, often referred to as the truncated matrix of $\mathbf{K}_N$.
The distribution of nonzero entries of $\widetilde{\mathbf{K}}_{121}$ is illustrated in the image (c) of Figure \ref{fig:coef}.
The truncation strategy leads to the fast Fourier-Galerkin method which is to find $\tilde{\boldsymbol{\rho}}_N$ that satisfies
\begin{equation}\label{truncated_matrix_equation}
	(\mathbf{I}_N-\widetilde{\mathbf{K}}_N)\tilde{\boldsymbol{\rho}}_N=\mathbf{g}_N.
\end{equation}

We denote the number of nonzero entries in matrix $\widetilde{\mathbf{K}}_N$ by $\mathcal{N} (\widetilde{\mathbf{K}}_N)$.

\begin{theorem}
Let $q>0$. Then there exists a positive constant $c$ such that for all $n \in\mathbb{N}$,
$\mathcal{N} (\widetilde{\mathbf{K}}_N)\leq c N\ln^2 N$, where $c$ only depends on $q$, and $N=(2n+1)^2$ is the order of $\widetilde{\mathbf{K}}_N$.
\end{theorem}

\begin{proof}
Given the definition of $\widetilde{\mathbf{K}}_N$, the number of its nonzero entries, $\mathcal{N} (\widetilde{\mathbf{K}}_N)$, does not exceed the number of entries in the set $\mathbb{L}_N(q)$.
For $q>0$, $n\in\mathbb{N}$ and $\mathbf{k}\in\mathbb{Z}_{-n,n}^2$, we define $\mathbb{S}_{\mathbf{k},N}(q):=\left\{\mathbf{l}\in\mathbb{Z}_{-n,n}^2:\|\mathbf{k}-\mathbf{l}\|_1\leq q\ln N\right\}$. It is clear that
$$
\mathbb{L}_N(q)=\left\{(\mathbf{k},\mathbf{l}): \mathbf{k} \in \mathbb{Z}_{-n,n}^2 {~\rm and~}\mathbf{l}\in \mathbb{S}_{\mathbf{k},N}(q)\right\}.
$$
Meanwhile, the inequality
$\|\mathbf{k}-\mathbf{l}\|_1\leq q\ln N$
indicates that $|k_\iota - l_\iota| \leq q\ln N$ for each $\iota \in \mathbb{Z}_2$.
This implies that for all $\mathbf{k} \in \mathbb{Z}_{-n,n}^2$, it holds that $\sum_{\mathbf{l} \in \mathbb{S}_{\mathbf{k},N}(q)} 1 \leq (2q\ln N+1)^2$. Consequently, we have
\begin{equation*}
\mathcal{N}(\widetilde{\mathbf{K}}_N) = \sum_{\mathbf{k} \in \mathbb{Z}_{-n,n}^2} \sum_{\mathbf{l} \in \mathbb{S}_{\mathbf{k},N}(q)} 1 \leq \sum_{\mathbf{k} \in \mathbb{Z}_{-n,n}^2} (2q\ln N+1)^2 = (2n+1)^2 (2q\ln N+1)^2.
\end{equation*}
Noting that $N = (2n+1)^2$, we achieve the desired result.
\end{proof}

\section{Analyzing the entries in $\mathbf{K}_N$}
\label{sec:3}

In this section, we explore the decay pattern of entries in $\mathbf{K}_N$ to evaluate the effectiveness of the proposed fast Fourier-Galerkin method. To facilitate this analysis, we define $G(\boldsymbol{\theta}, \boldsymbol{\vartheta}) := K(\boldsymbol{\theta}, \boldsymbol{\theta} + \boldsymbol{\vartheta})$ for all $\boldsymbol{\theta}, \boldsymbol{\vartheta} \in I^2_{2\pi}$. For a function $\varphi$ defined on $I^2_{2\pi}$ and for all $\mathbf{k} \in \mathbb{Z}^2$, we define the Fourier coefficient of $\varphi$ by
$
\hat{\varphi}_\mathbf{k}:=\int_{I^2_{2\pi}}\varphi(\boldsymbol{\theta})e_{-\mathbf{k}}(\boldsymbol{\theta})d\boldsymbol{\theta}.
$
Similarly, for a function $\Phi$ defined on $I^4_{2\pi}$ and for all $\mathbf{k}, \mathbf{l} \in \mathbb{Z}^2$, the Fourier coefficient of $\Phi$ is given by
$$
\hat{\Phi}_{\mathbf{k},\mathbf{l}}:=\int_{I^2_{2\pi}}\int_{I^2_{2\pi}}\Phi(\boldsymbol{\theta},\boldsymbol{\eta})e_{-\mathbf{k}}(\boldsymbol{\theta})e_{-\mathbf{l}}(\boldsymbol{\phi})d\boldsymbol{\theta}d\boldsymbol{\eta}.
$$
Given the periodicity of the kernel function $K$, it is evident that for all $\mathbf{k}, \mathbf{l} \in \mathbb{Z}^2$, the relationship $\hat{G}_{\mathbf{k}, \mathbf{l}} = \hat{K}_{\mathbf{k} - \mathbf{l}, \mathbf{l}} = K_{\mathbf{k} - \mathbf{l}, -\mathbf{l}}$ holds. By analyzing the decay pattern of the Fourier coefficients of $G$, we can thus determine the decay pattern of the entries in $\mathbf{K}_N$.

The definition of $G$ shows that $G$ is not singular at any point $(\boldsymbol{\theta},\boldsymbol{\vartheta})\in I^4_{2\pi}$ where $\zeta(\boldsymbol{\vartheta})\neq 0$. Moreover, when the functions $\gamma_j$, for $j\in\mathbb{Z}_3$, exhibit holomorphic properties in a domain of $\mathbb{C}^2$, for given $\boldsymbol{\vartheta}\in I_{2\pi}^2$ satisfying $\zeta(\boldsymbol{\vartheta})\neq 0$, the function $G(\cdot, \boldsymbol{\vartheta})$ can be analytically extended to a certain domain in $\mathbb{C}^2$. This raises a crucial question: \emph{How does the analytic extension domain of the function $G(\cdot, \boldsymbol{\vartheta})$ vary with changes in $\boldsymbol{\vartheta}$}?
Understanding this variation will help analyze the decay pattern of the Fourier coefficients of
$G$. To answer the above question, we next introduce a concept called uniformly analytic extension.
Utilizing this concept, we then proceed to analyze the decay patterns of both the Fourier coefficients of $G$ and the entries in  $\mathbf{K}_N$.

\subsection{Uniformly analytic extension}

Define $\mathbb{W}(r):=\{z=x+iy:x\in\mathbb{R},-r<y<r\}$ for all $r>0$.
For a given function $\Phi$ defined on $\mathbb{R}^2\otimes I_{2\pi}^2$ and $r>0$, if there is a function $\tilde{\Phi}$ defined on $\mathbb{W}(r)^2\otimes I_{2\pi}^2$ such that for all $\boldsymbol{\vartheta}\in I_{2\pi}^2$ with $\zeta(\boldsymbol{\vartheta})\neq 0$, $\tilde{\Phi}(\cdot,\boldsymbol{\vartheta})$ is holomorphic on $\mathbb{W}(r)^2$ and satisfies $\tilde{\Phi}(\cdot,\boldsymbol{\vartheta})=\Phi(\cdot,\boldsymbol{\vartheta})$ on $\mathbb{R}^2$, then we say $\tilde{\Phi}$ is the uniformly analytic extension of $\Phi$, and $\Phi$ can be uniformly analytic extended to $\mathbb{W}(r)^2$ with respect to $\boldsymbol{\theta}$.
To simplify notations in this article, $\Phi$ will also denote its uniformly analytic extension $\tilde{\Phi}$. As part of our preparation for analyzing the decay pattern of the Fourier coefficients of $G$, we will next study the Fourier coefficients of $\Phi$ which can be uniformly analytic extended to $\mathbb{W}(r)^2$ with respect to $\boldsymbol{\theta}$.
	
The notation $|\cdot|$ denotes the complex modulus.
Define
$$
\mathbb{U}(r):=\{z=x+iy:x\in I_{2\pi},-r<y <r\}
$$
and
$$\mathbb{V}(r):= \{w\in\mathbb{C}:e^{-r}<|w|<e^{r}\}
$$
for all $r>0$.
Let $\mu(z):=e^{iz}$ for $z\in \mathbb{U}(r)$, which bijectively maps $\mathbb{U}(r)$ to $\mathbb{V}(r)$.
The inverse of $\mu$ is denoted by $\mu^{-1}$.
For any $0<\delta<r$, $\boldsymbol{\vartheta}\in I_{2\pi}^2$ with $\zeta(\boldsymbol{\vartheta})\neq 0$, and a function $\Phi$ defined on $\mathbb{W}(r)^2\otimes I_{2\pi}^2$ that is continuous on $\mathbb{W}(r)^2$, we define $S_{\delta,\Phi}(\boldsymbol{\vartheta}):=2\pi\sup\{| \Phi(\mathbf{z},\boldsymbol{\vartheta})|:\mathbf{z}=(z_0,z_1)\in \mathbb{C}^2,{~\rm and~}|\mathrm{Im}(z_0)|=|\mathrm{Im}(z_1)|=\delta \}$.
For any $\sigma>0$, define $\tau_\sigma:=\{w\in\mathbb{C}:|w|=\sigma \}$.

\begin{lemma}\label{bivarite_four_coe}
Let $\Phi$ be a real-valued function define on $\mathbb{R}^2\otimes I_{2\pi}^2$, and $r>0$.
If $\Phi$ can be uniformly analytic extended to $\mathbb{W}(r)^2$ with respect to $\boldsymbol{\theta}$, and function $\Phi(\cdot, \boldsymbol{\vartheta})$ is $2\pi$-biperiodic for any $\boldsymbol{\vartheta}\in I_{2\pi}^2$ with $\zeta(\boldsymbol{\vartheta})\neq 0$,
then for all $0<\delta<r$, $\boldsymbol{\vartheta}\in I_{2\pi}^2$ with $\zeta(\boldsymbol{\vartheta})\neq 0$ and  $\mathbf{k}\in \mathbb{Z}^2$, there is
\begin{equation*}\label{part_four_ineq}
	\left| \widehat{ (\Phi(\cdot,\boldsymbol{\vartheta}))}_{\mathbf{k}}\right|\leq S_{\delta,\Phi}(\boldsymbol{\vartheta})e^{-\delta\Vert\mathbf{k}\Vert_1}.
\end{equation*}
\end{lemma}
	
\begin{proof}
Let $0<\delta<r$, $\boldsymbol{\vartheta}\in I_{2\pi}^2$ with $\zeta(\boldsymbol{\vartheta})\neq 0$ and $\mathbf{k}:=[k_0,k_1]\in \mathbb{Z}^2$.
The conditions of this lemma ensure that $\Phi(\cdot,\boldsymbol{\vartheta})$ is holomorphic  and $2\pi$-biperiodic on $\mathbb{W}(r)^2$. Moreover, since $\mu^{-1}$ is a holomorphic mapping from $\mathbb{V}(r)$ to $\mathbb{U}(r)\subset \mathbb{W}(r)$, the function $F_{\boldsymbol{\vartheta}}$ defined by
$F_{\boldsymbol{\vartheta}}(w_0,w_1):= \Phi(\mu^{-1}(w_0),\mu^{-1}(w_1),\boldsymbol{\vartheta})$, $[w_0,w_1]\in \mathbb{V}(r)^2$,
is also holomorphic on $\mathbb{V}(r)^2$.
Therefore, it follows from Theorem 1.118 in \cite{Scheidemann2023Complex} that the function $F_{\boldsymbol{\vartheta}}$ can be represented as the following Laurent series,
\begin{equation}\label{Laurent_series}
F_{\boldsymbol{\vartheta}}(w_0,w_1)=\sum_{l_0\in\mathbb{Z}}\sum_{l_1\in\mathbb{Z}}a_{l_0,l_1}(\boldsymbol{\vartheta})w_0^{l_0}w_1^{l_1},\quad [w_0,w_1]\in \mathbb{V}(r)^2,
\end{equation}
with the coefficients
\begin{equation}\label{Laurent_coef}
a_{l_0,l_1}(\boldsymbol{\vartheta}):=\frac{1}{(2\pi i)^2}\oint_{\tau_{\sigma_0}}\oint_{\tau_{\sigma_1}}\frac{F_{\boldsymbol{\vartheta}}(\xi_0,\xi_1)}{\xi_0^{l_0+1}\xi_1^{l_1+1}}d \xi_1 d\xi_0,
\end{equation}
where $\sigma_0$ and $\sigma_1$ can be arbitrary values in   the interval $(e^{-r }, e^{r})$.
		
Setting $l_0=k_0$, $l_1=k_1$ and $\sigma_0=\sigma_1=1$, the change of variables $\xi_0=e^{it_0}$ and $\xi_1=e^{it_1}$ in \eqref{Laurent_coef} results in the following expression
\begin{equation}\label{eq_akk}
	a_{k_0,k_1}(\boldsymbol{\vartheta})=\frac{1}{(2\pi )^2}\int_{I_{2\pi}}\int_{I_{2\pi}}\frac{F_{\boldsymbol{\vartheta}}(e^{it_0},e^{it_1})}{e^{it_0k_0+it_1k_1}}d t_1 dt_0.
\end{equation}
The definitions of $\mu$ and $F_{\boldsymbol{\vartheta}}$ ensure that for all $\mathbf{t}=[t_0,t_1]\in I_{2\pi}^2$, there is $F_{\boldsymbol{\vartheta}}(e^{it_0},e^{it_1})= \Phi(\mathbf{t},\boldsymbol{\vartheta})$. Substituting this equality into \eqref{eq_akk} leads to $(\widehat{ \Phi(\cdot,\boldsymbol{\vartheta})})_\mathbf{k}=2\pi  a_{k_0,k_1}(\boldsymbol{\vartheta})$.
Furthermore, from \eqref{Laurent_coef}, we can infer that for all $e^{-r }<\sigma_0< e^{r}$ and $e^{-r }<\sigma_1< e^{r}$,
\begin{equation*}
(\widehat{\Phi(\cdot,\boldsymbol{\vartheta})})_\mathbf{k}=-\frac{1}{2\pi}\oint_{\tau_{\sigma_0}}\oint_{\tau_{\sigma_1}}\frac{F_{\boldsymbol{\vartheta}}(\xi_0,\xi_1)}{
\xi_0^{k_0+1}\xi_1^{k_1+1}}d \xi_1 d\xi_0.
\end{equation*}
Therefore, applying the ML inequality (see Proposition 1.17(b) in \cite{Conway1978}) to the right hand side of the above equality demonstrates that for all $e^{-r }<\sigma_0< e^{r}$ and $e^{-r }<\sigma_1< e^{r}$,
\begin{equation}\label{ML_ineq}
|(\widehat{ \Phi(\cdot,\boldsymbol{\vartheta})})_\mathbf{k}|\leq{2\pi}M_{\sigma_0,\sigma_1}(\boldsymbol{\vartheta})\sigma_0^{-k_0}\sigma_1^{-k_1},
\end{equation}
where $M_{\sigma_0,\sigma_1}(\boldsymbol{\vartheta}):=\sup\{|F_{\boldsymbol{\vartheta}}(w_0,w_1)|:w_\iota\in\tau_{\sigma_\iota},\iota\in\mathbb{Z}_2\}$.
Note that $\mu$ bijectively maps $\{x+iy:x\in I_{2\pi} {\rm~and~} y=-\ln\sigma\}$ to $\tau_\sigma$ for all $e^{-r }<\sigma< e^{r}$. Hence, from the definitions of $F_{\boldsymbol{\vartheta}}$ and $M_{\sigma_0,\sigma_1}(\boldsymbol{\vartheta})$, we establish that for all $e^{-r }<\sigma_0< e^{r}$ and $e^{-r }<\sigma_1< e^{r}$,
\begin{eqnarray}\label{def_rho_j}
	\nonumber  M_{\sigma_0,\sigma_1}(\boldsymbol{\vartheta})&=&\sup\{| \Phi(\mathbf{z},\boldsymbol{\vartheta})|:\mathbf{z}=[z_0,z_1]\in\mathbb{C}^2, \mathrm{Im}(z_\iota)=-\ln\sigma_\iota,\iota\in\mathbb{Z}_2\}\\
	&\leq& \sup\{| \Phi(\mathbf{z},\boldsymbol{\vartheta})|:\mathbf{z}=[z_0,z_1]\in\mathbb{C}^2, |\mathrm{Im}(z_\iota)|=|\ln\sigma_\iota|,\iota\in\mathbb{Z}_2\}.
\end{eqnarray}
Therefore, by setting $\sigma_\iota=e^{\delta}$ when $k_\iota\geq 0$ and $\sigma_\iota=e^{-\delta}$ when $k_\iota< 0$, and substituting \eqref{def_rho_j} into \eqref{ML_ineq}, we achieve the desired result based on the definition of $S_{\delta,\Phi}(\boldsymbol{\vartheta})$.
\end{proof}

Next, we establish an upper bound for the modulus of the Fourier coefficients of $\Phi$ which meets the conditions specified in Lemma \ref{bivarite_four_coe}.
	
\begin{theorem}\label{four_decay_general}
Let $\Phi$ be a real-valued function define on $\mathbb{R}^2\otimes I_{2\pi}^2$, and $r>0$.
%Let $r>0$ and $\Phi$ be a real-valued function defined on $\mathbb{R}^2\otimes I_{2\pi}^2$.
If function $\Phi$ satisfies the conditions outlined in Lemma \ref{bivarite_four_coe}, and there exists $M>0$ such that for all $0<\delta<r$,
\begin{equation}\label{eq_four_decay_general0}
	\int_{I_{2\pi}^2} S_{\delta,\Phi}(\boldsymbol{\vartheta})d\boldsymbol{\vartheta}\leq M,
\end{equation}
then for all $0<\delta<r$ and $\mathbf{k}, \mathbf{l}\in\mathbb{Z}^2$, there holds
$\left|\hat{\Phi}_{\mathbf{k},\mathbf{l}}\right|\leq\frac{M}{2\pi}e^{-\delta\Vert\mathbf{k}\Vert_1}$.
\end{theorem}
\begin{proof}
By the definitions of Fourier coefficients and Lemma \ref{bivarite_four_coe}, we know that for all $0<\delta<r$ and $\mathbf{k}, \mathbf{l}\in\mathbb{Z}^2$,
\begin{eqnarray*}
\left|\hat{\Phi}_{\mathbf{k},\mathbf{l}}\right|&=&\left|\int_{I_{2\pi}^2}\left(\int_{I_{2\pi}^2}\Phi( \boldsymbol{\theta}, \boldsymbol{\vartheta})e_{-\mathbf{k}}(\boldsymbol{\theta})d \boldsymbol{\theta}\right)e_{- \mathbf{l}}(\boldsymbol{\vartheta})d \boldsymbol{\vartheta}\right|\\
&\leq& \frac{1}{2\pi}\int_{I_{2\pi}^2}  \left|\widehat{ (\Phi(\cdot,\boldsymbol{\vartheta}))}_{\mathbf{k}} \right|d \boldsymbol{\vartheta}\leq \frac{e^{-\delta\|\mathbf{k}\|_1}}{2\pi}\int_{I_{2\pi}^2} S_{\delta,\Phi}(\boldsymbol{\vartheta})d \boldsymbol{\vartheta}.
\end{eqnarray*}
Then, by substituting \eqref{eq_four_decay_general0} into the above inequality, we obtain the result of this theorem.
\end{proof}

We have demonstrated that the conditions specified in Theorem \ref{four_decay_general} lead to the exponential decay property of Fourier coefficients for real-valued functions defined on $\mathbb{R}^2\otimes I_{2\pi}^2$.
In the next subsection, we turn our attention to verifying that the function $G$ satisfies these conditions, and show the decay pattern of ${K}_{\mathbf{k},\mathbf{l}}$.

\subsection{Decay pattern of ${K}_{\mathbf{k},\mathbf{l}}$}
Define the distance function $d(\boldsymbol{\theta}, \boldsymbol{\vartheta}):=\|\Gamma(\boldsymbol{\theta})-\Gamma(\boldsymbol{\theta}+\boldsymbol{\vartheta})\|^2_2$ for all $\boldsymbol{\theta},\boldsymbol{\vartheta}\in I^2_{2\pi}$.
Combining \eqref{kernel_K_specific} and the definition of $G$, we can rewrite $G$ into the form as
\begin{equation}\label{eq_Gd}
	G(\boldsymbol{\theta}, \boldsymbol{\vartheta})= -\frac{1}{2\pi}\frac{(\Gamma( \boldsymbol{\theta})-\Gamma( \boldsymbol{\theta}+\boldsymbol{\vartheta}))\cdot \left(\frac{\partial\Gamma}{\partial \theta_0}\times\frac{\partial\Gamma}{\partial \theta_1}\right)(\boldsymbol{\theta}+\boldsymbol{\vartheta})}{d^{3/2}(\boldsymbol{\theta},\boldsymbol{\vartheta})}.
\end{equation}
This formulation clearly indicates that for a fixed $\boldsymbol{\vartheta}$, where the elements of $\Gamma$ are analytic and the analytic continuation of the function $d(\cdot, \boldsymbol{\vartheta})$ is non-zero, the function $G(\cdot, \boldsymbol{\vartheta})$ is holomorphic.
Consequently, to determine the region for the analytic continuation of $G(\cdot, \boldsymbol{\vartheta})$, we investigate regions in $\mathbb{C}^2$ where the continuation of $d(\cdot, \boldsymbol{\vartheta})$ remains non-zero.

To achieve this aim, we introduce an additional assumption regarding the boundary $\Gamma$, and define some notations. For a given region $B\subset\mathbb{C}^2$, the closure of $B$ is denoted by $\overline{B}$. Let $R_0>0$.
We assume that
	
\noindent	\textbf{(A2)} For all $j\in\mathbb{Z}_3$, the periodic extension of function $\gamma_j$ can be analytic extended to $\overline{W(R_0)^2}$.

The complex derivative of $ \Gamma $ is denoted by $D\Gamma$. For all $\varepsilon>0$, there exists a $ \delta>0 $ such that for all $ \mathbf{z}\in\mathbb{C}^2 $ and $ \mathbf{a}\in\mathbb{C}^2 $ with $ \|\mathbf{a}\|_2<\delta$,
$\| \Gamma(\mathbf{z}+\mathbf{a})-\Gamma(\mathbf{z})-(D\Gamma)(\mathbf{z})\mathbf{a} \|_2 \le \varepsilon\|\mathbf{a}\|_2$ holds (see Definition 1.22 in \cite{Scheidemann2023Complex} for details).
For an $m\times m$ matrix $\mathbf{B}$, we let $ \|\mathbf{B}\|_2:=\max\{\|\mathbf{B}\mathbf{z}\|_2:\|\mathbf{z}\|_2=1\} $.
%From the definition of $ \|\mathbf{B}\|_2 $, we can obtain that $ 	\Vert\mathbf{B}\Vert_2=(\max\{\Lambda(\mathbf{B}^T\mathbf{B})\})^{1/2}$, where $\Lambda(\mathbf{B}^T\mathbf{B})$ is the set of the eigenvalues of matrix $\mathbf{B}^T\mathbf{B}$ and $\mathbf{B}^T$ is the transpose of $\mathbf{B}$.
For each $j\in\mathbb{Z}_3$, let $\mathbf{H}_j(\mathbf{z}):=(D(D\gamma_j(\mathbf{z}))^T)^T$. Define $ M_{0,j}$ and $ M_{1,j}$ as $\sup\left\{\|D\gamma_j(\mathbf{z})+D\gamma_j(\mathrm{Re}(\mathbf{z}))\|_2:\mathbf{z}\in \mathbb{W}(R_0)^2\right\}$ and $\sup\left\{\|\mathbf{H}_j(\mathbf{z})\|_2:\mathbf{z}\in \mathbb{W}(R_0)^2\right\}$, respectively.
We also define $ M_{0}:=\sqrt{\sum_{j\in\mathbb{Z}_3}M_{0,j}^2}$ and $ M_{1}:=\sqrt{\sum_{j\in\mathbb{Z}_3}M_{1,j}^2}$.
Let $b$ be a real number satisfying
\begin{equation}\label{eq_b}
0 < b < \min\left\{\frac{C_0^2}{\sqrt{2}M_0M_1},\frac{R_0}{2}\right\},
\end{equation}
and define $b^*:=\frac{1}{2}\left(b+\min\left\{\frac{C_0^2}{\sqrt{2}M_0M_1},\frac{R_0}{2}\right\}\right)$ and $ C_1:=C_0^2-\sqrt{2}M_0M_1b^* $.
For all $\boldsymbol{\vartheta}\in I_{2\pi}^2$, let
$\check{\boldsymbol{\vartheta}} :=[\check{\vartheta}_0,\check{\vartheta}_1]$ with
\begin{equation*}
	\check{\vartheta}_{\iota}:=\begin{cases}
		\vartheta_{\iota},&\vartheta_{\iota} \in [0,\pi),\\
		\vartheta_{\iota}-2\pi, &\vartheta_{\iota} \in [\pi,2\pi),
	\end{cases}\quad {\rm for~}\iota\in\mathbb{Z}_2.
\end{equation*}

\begin{lemma}\label{tilde_d_lwrbnd_Su}
If assumptions \textup{\textbf{(A1)}} and \textup{\textbf{(A2)}} hold, then  for all $\mathbf{z}\in \mathbb{W}(b^*)^2$ and $\boldsymbol{\vartheta}\in I_{2\pi}^2$, the following inequality holds
\begin{equation}\label{lemma_7_tilde_d_cplx_lwbnd_Su}
\left\vert d (\mathbf{z},\boldsymbol{\vartheta}) \right\vert \geq C_1(\zeta(\boldsymbol{\vartheta}))^2.
\end{equation}
\end{lemma}
\begin{proof}
Let $\mathbf{z}:=[z_0,z_1]\in \mathbb{W}(b^*)^2$ and $\boldsymbol{\vartheta}:=[\vartheta_0,\vartheta_1]\in I_{2\pi}^2$. Denote $\mathbf{z}_{\mathrm{R}}:=[\mathrm{Re}(z_0),\mathrm{Re}(z_1)]$. Applying the triangle inequality yields
\begin{equation}\label{lemmaSu2_tri_ineq}
	\left\vert d(\mathbf{z},\boldsymbol{\vartheta}) \right\vert
	\geq \left\vert d(\mathbf{z}_{\mathrm{R}},\boldsymbol{\vartheta})\right\vert - \left\vert d(\mathbf{z},\boldsymbol{\vartheta})-d(\mathbf{z}_{\mathrm{R}},\boldsymbol{\vartheta})\right\vert,
\end{equation}
where
\begin{eqnarray}\label{lemmaSu2_total}
\nonumber |d(\mathbf{z},\boldsymbol{\vartheta})-d(\mathbf{z}_{\mathrm{R}},\boldsymbol{\vartheta})&=&\left|\sum_{j\in\mathbb{Z}_3}\left(\gamma_j(\mathbf{z}+\boldsymbol{\vartheta})-\gamma_j(\mathbf{z})\right)^2-\sum_{j\in\mathbb{Z}_3}
\left(\gamma_j(\mathbf{z}_{\mathrm{R}}+\boldsymbol{\vartheta})-\gamma_j(\mathbf{z}_{\mathrm{R}})\right)^2\right|\\
&\le&\sum_{j\in\mathbb{Z}_3}\left|\gamma_j(\mathbf{z}+\boldsymbol{\vartheta})+\gamma_j(\mathbf{z}_{\mathrm{R}}+\boldsymbol{\vartheta})-\gamma_j(\mathbf{z})-\gamma_j(\mathbf{z}_{\mathrm{R}}) \right|
 | \gamma_j(\mathbf{z}+\boldsymbol{\vartheta})\\\nonumber&-&\gamma_j(\mathbf{z}_{\mathrm{R}}+\boldsymbol{\vartheta})-\gamma_j(\mathbf{z})+\gamma_j(\mathbf{z}_{\mathrm{R}}) |.			
\end{eqnarray}
		
We now discuss the terms of the right hand side in \eqref{lemmaSu2_total}. Let $j\in\mathbb{Z}_3$.
The periodicity of $\gamma_j$ ensures that $\gamma_j(\mathbf{z}_{\mathrm{R}}+\boldsymbol{\vartheta})=\gamma_j(\mathbf{z}_{\mathrm{R}}+\check{\boldsymbol{\vartheta}})$. Consequently, there is
\begin{equation*}
\gamma_j(\mathbf{z}+\boldsymbol{\vartheta})+\gamma_j(\mathbf{z}_{\mathrm{R}}+\boldsymbol{\vartheta})-\gamma_j(\mathbf{z})-\gamma_j(\mathbf{z}_{\mathrm{R}})
=(\gamma_j(\mathbf{z}+\check{\boldsymbol{\vartheta}})+\gamma_j(\mathbf{z}_{\mathrm{R}}+\check{\boldsymbol{\vartheta}}))-(\gamma_j(\mathbf{z})+\gamma_j(\mathbf{z}_{\mathrm{R}})).	
\end{equation*}
Then, utilizing the differential inequality (refer to \cite{mcleod1965mean}) leads to
$$
|\gamma_j(\mathbf{z}+\boldsymbol{\vartheta})+\gamma_j(\mathbf{z}_{\mathrm{R}}+\boldsymbol{\vartheta})-\gamma_j(\mathbf{z})-\gamma_j(\mathbf{z}_{\mathrm{R}})| \le \left\|D\gamma_j(\mathbf{z}+\lambda_{0,j}\check{\boldsymbol{\vartheta}}) +D\gamma_j(\mathbf{z}_{\mathrm{R}}+\lambda_{0,j}\check{\boldsymbol{\vartheta}})\right\|_2\left\|\check{\boldsymbol{\vartheta}}\right\|_2,
$$
where $ \lambda_{0,j}\in [0,1]$.
From the definition of $M_{0,j}$, it follows that
\begin{equation}\label{lemmaSu2_First_ineq}
\left|(\gamma_j(\mathbf{z}+\boldsymbol{\vartheta})+\gamma_j(\mathbf{z}_{\mathrm{R}}+\boldsymbol{\vartheta}))-(\gamma_j(\mathbf{z})+\gamma_j(\mathbf{z}_{\mathrm{R}})) \right|\le M_{0,j}\left\|\check{\boldsymbol{\vartheta}}\right\|_2.
\end{equation}
Define $\mathbf{z}_{\mathrm{I}}:=[\mathrm{Im}(z_0),\mathrm{Im}(z_1)]$, implying that $ \mathbf{z}=\mathbf{z}_{\mathrm{R}}+i\mathbf{z}_{\mathrm{I}}$ and $\mathbf{z}_{\mathrm{I}}\in (-b^*,b^*)^2$. Then, with the periodicity of $\gamma_j$, we have that
$$
\gamma_j(\mathbf{z}+\boldsymbol{\vartheta})-\gamma_j(\mathbf{z}_{\mathrm{R}}+\boldsymbol{\vartheta})-\gamma_j(\mathbf{z})+\gamma_j(\mathbf{z}_{\mathrm{R}})
=(\gamma_j(\mathbf{z}_{\mathrm{R}}+i\mathbf{z}_{\mathrm{I}}+\check{\boldsymbol{\vartheta}})-\gamma_j(\mathbf{z}_{\mathrm{R}}+\check{\boldsymbol{\vartheta}}))
-(\gamma_j(\mathbf{z}_{\mathrm{R}}+i\mathbf{z}_{\mathrm{I}})-\gamma_j(\mathbf{z}_{\mathrm{R}})).
$$
Applying the differential inequality to this equality yields
\begin{eqnarray*}
\left|\gamma_j(\mathbf{z}+\boldsymbol{\vartheta})-\gamma_j(\mathbf{z}_{\mathrm{R}}+\boldsymbol{\vartheta})-\gamma_j(\mathbf{z})+\gamma_j(\mathbf{z}_{\mathrm{R}})\right|
& \le &\left\| D\gamma_j(\mathbf{z}+\lambda_{1,j}\check{\boldsymbol{\vartheta}})-D\gamma_j(\mathbf{z}_{\mathrm{R}}+\lambda_{1,j}\check{\boldsymbol{\vartheta}})\right\|_2 \left\|\check{\boldsymbol{\vartheta}}\right\|_2\\
&\le& \left\|\mathbf{z}_{\mathrm{I}}\right\|_2 \left\|\mathbf{H}_j(\mathbf{z}_{\mathrm{R}}+\lambda_{1,j}^{\mathrm{R}}\boldsymbol{\vartheta}+i\lambda_{1,j}^{\mathrm{I}}\mathbf{z}_{\mathrm{I}})\right\|_2\left\|\check{\boldsymbol{\vartheta}}\right\|_2,
\end{eqnarray*}
where $ \lambda_{1,j}^{\mathrm{R}},\lambda_{1,j}^{\mathrm{I}}\in [0,1] $.
Thus, according to the definition of $M_{1,j}$, we obtain that
\begin{equation}\label{lemmaSu2_Second_ineq}
\left|(\gamma_j(\mathbf{z}+\boldsymbol{\vartheta})-\gamma_j(\mathbf{z}_{\mathrm{R}}+\boldsymbol{\vartheta}))-(\gamma_j(\mathbf{z})-\gamma_j(\mathbf{z}_{\mathrm{R}}))\right|\le M_{1,j}\|\mathbf{z}_{\mathrm{I}}\|_2 \left\|\check{\boldsymbol{\vartheta}}\right\|_2.
\end{equation}

Note that $ \left\|\check{\boldsymbol{\vartheta}}\right\|_2=\zeta(\boldsymbol{\vartheta}) $. Substituting \eqref{lemmaSu2_First_ineq} and \eqref{lemmaSu2_Second_ineq} into  \eqref{lemmaSu2_total}, and considering the definitions of $M_0$ and $M_1$, we derive
\begin{equation}\label{lemmaSu2_total_ineq}
\left|d(\mathbf{z},\boldsymbol{\vartheta})-d(\mathbf{z}_{\mathrm{R}},\boldsymbol{\vartheta})\right|\le(\zeta(\boldsymbol{\vartheta}))^2\|\mathbf{z}_{\mathrm{I}}\|_2\sum_{j\in\mathbb{Z}_3}M_{0,j}M_{1,j}\le M_0M_1(\zeta(\boldsymbol{\vartheta}))^2\|\mathbf{z}_{\mathrm{I}}\|_2.
\end{equation}
Meanwhile, assumption \textup{\textbf{(A1)}} and the periodicity of $d$ imply
\begin{equation}\label{lemmaSu2_assump_a1_d}
d(\mathbf{z}_{\mathrm{R}},\boldsymbol{\vartheta}) \geq C_0^2  (\zeta(\boldsymbol{\vartheta}))^2.
\end{equation}
Substituting \eqref{lemmaSu2_total_ineq} and \eqref{lemmaSu2_assump_a1_d} into the right hand side of \eqref{lemmaSu2_tri_ineq} yields
\begin{equation}\label{lemmaSu2_assump_a1_d2}
\left\vert d(\mathbf{z},\boldsymbol{\vartheta}) \right\vert\geq \left(C_0^2-M_0M_1\|\mathbf{z}_{\mathrm{I}}\|_2\right)(\zeta(\boldsymbol{\vartheta}))^2.
\end{equation}
Note that for all $\mathbf{z}\in \mathbb{W}(b^*)^2$, there is $ \|\mathbf{z}_{\mathrm{I}}\|_2\le\sqrt{2}b^* $.
Thus, according to the definitions of $ b^* $ and $ C_1 $, for all $ \mathbf{z}\in \mathbb{W}(b^*)^2 $, there holds
$
C_0^2-M_0M_1\|\mathbf{z}_{\mathrm{I}}\|_2 \ge C_0^2-\sqrt{2}M_0M_1b^*=C_1$.
Therefore, from \eqref{lemmaSu2_assump_a1_d2}, we obtain our desired result \eqref{lemma_7_tilde_d_cplx_lwbnd_Su}.
\end{proof}
	
To verify that the function $G$ meets the conditions in Theorem \ref{four_decay_general}, it remains to establish a bound for the analytic continuation of the numerator in the definition of $G$ given in \eqref{eq_Gd}. To facilitate this, we introduce the following notations. Denote the partial derivatives of $\Gamma$ in the complex domain $\mathbb{W}(b^*)^2$ as $\frac{\partial\Gamma}{\partial z_\iota}$, $\iota\in\mathbb{Z}_2$. Specifically, $\frac{\partial\Gamma}{\partial z_\iota}$ is defined as
$\frac{\partial\Gamma}{\partial z_\iota}:=\left[\frac{\partial\gamma_0}{\partial z_\iota}, \frac{\partial\gamma_1}{\partial z_\iota}, \frac{\partial\gamma_2}{\partial z_\iota}\right]^T $. It can be check that $D\Gamma=\left[\frac{\partial \Gamma}{\partial z_0},\frac{\partial \Gamma}{\partial z_1}\right]$ (see Exercise 1.47, problem 2 in \cite{Scheidemann2023Complex}).
Let $ \mathbb{Z}_+ := \{0,1,2,\cdots\}$. For each $ \mathbf{k}:=[k_0,k_1]\in\mathbb{Z}_+^2 $, we define $ \mathbf{k}+1:=[k_0+1,k_1+1]$ and $ \mathbf{k}!:=k_0!k_1! $. For any $ \mathbf{z}\in\mathbb{C}^2 $ and $ \mathbf{k}\in\mathbb{Z}_+^2 $, we specify $ \mathbf{z}^{\mathbf{k}}:=z_0^{k_0}z_1^{k_1} $ and $ \frac{\partial^{\mathbf{k}}}{\partial \mathbf{z}^\mathbf{k}}:=\frac{\partial^{\|\mathbf{k}\|_1}}{\partial z_0^{k_0}\partial z_1^{k_1}} $.

\begin{lemma}\label{lemma_upper_numer_G}
If assumptions \textup{\textbf{(A1)}} and \textup{\textbf{(A2)}} hold, then there exists a positive constant $c$ such that for all $\mathbf{z}\in \mathbb{W}(b^*)^2$ and $\boldsymbol{\vartheta}\in I_{2\pi}^2$,
\begin{equation}\label{Gamma_ineq}
\left|(\Gamma(\mathbf{z})-\Gamma(\mathbf{z}+\boldsymbol{\vartheta}))\cdot \left(\dfrac{\partial\Gamma}{\partial z_0}\times\dfrac{\partial\Gamma}{\partial z_1}\right)(\mathbf{z}+\boldsymbol{\vartheta})\right|\leq c (\zeta(\boldsymbol{\vartheta}))^2.
\end{equation}
\end{lemma}
\begin{proof}
Let $0<\beta<1$ such that $\beta R_0<\pi$. Define $\tilde{I}_\beta:=[0,\beta R_0)\cup(2\pi-\beta R_0, 2\pi)$.
Note that for $\mathbf{z}\in \mathbb{W}(b^*)^2$ and $\boldsymbol{\vartheta}\in I_{2\pi}^2\backslash \tilde{I}_{\beta}^2$, we have $0<\beta R_0\le \|\check{\boldsymbol{\vartheta}}\|_2=\zeta(\boldsymbol{\vartheta}) \le \pi$. It follows that
$$\frac{\left|(\Gamma(\mathbf{z})-\Gamma(\mathbf{z}+\boldsymbol{\vartheta}))\cdot \left(\frac{\partial\Gamma}{\partial z_0}\times\frac{\partial\Gamma}{\partial z_1}\right)(\mathbf{z}+\boldsymbol{\vartheta})\right|}{\zeta(\boldsymbol{\vartheta})^2} \leq \frac{\left|(\Gamma(\mathbf{z})-\Gamma(\mathbf{z}+{\boldsymbol{\vartheta}}))\cdot \left(\frac{\partial\Gamma}{\partial z_0}\times\frac{\partial\Gamma}{\partial z_1}\right)(\mathbf{z}+{\boldsymbol{\vartheta}})\right|}{(\beta R_0)^2}.$$
Hence,  from the continuity and differentiability of $\Gamma$, we know that there is a positive constant $c_0$ such that for all $\mathbf{z}\in \mathbb{W}(b^*)^2$ and $\boldsymbol{\vartheta}\in I_{2\pi}^2\backslash \tilde{I}_{\beta}^2$, inequality \eqref{Gamma_ineq} holds.

To prove this lemma, it remains to show that there exists a positive constant $c_2$ such that for all $\mathbf{z}\in \mathbb{W}(b^*)^2$ and $\boldsymbol{\vartheta}\in \tilde{I}_{\beta}^2$, inequality \eqref{Gamma_ineq} also holds. We begin by reformulating the left hand side of \eqref{Gamma_ineq}.
Since $b^*<R_0$, assumption \textup{\textbf{(A2)}} confirms that for each $j\in\mathbb{Z}_3$, the functions $\frac{\partial^{\mathbf{k}}\gamma_j}{\partial \mathbf{z}^\mathbf{k}}$, $\mathbf{k}\in\mathbb{Z}_+^2$, are $2\pi$ bi-periodic with respect to the real part over $\mathbb{W}(b^*)^2$. That is $\frac{\partial^{\mathbf{k}}\gamma_j}{\partial \mathbf{z}^\mathbf{k}}(\mathbf{z})=\frac{\partial^{\mathbf{k}}\gamma_j}{\partial \mathbf{z}^\mathbf{k}}(z_0+2\pi\iota_0,z_1+2\pi\iota_1)$ for all $\mathbf{z}:=[z_0,z_1]\in \mathbb{W}(b^*)^2$ and $\iota_0,\iota_1\in \mathbb{Z}$.
Therefore, by noting that $\mathbf{z}+\check{\boldsymbol{\vartheta}}\in \mathbb{W}(b^*)^2$ for all $\mathbf{z}\in \mathbb{W}(b^*)^2$ and $\boldsymbol{\vartheta}\in \tilde{I}_{\beta}^2$, utilizing Taylor's expansion
yields
\begin{equation*}
\gamma_j(\mathbf{z})=\sum_{\mathbf{k}\in\mathbb{Z}_+^2}\dfrac{1}{\mathbf{k}!}\frac{\partial^{\mathbf{k}}\gamma_j}{\partial \mathbf{z}^\mathbf{k}}(\mathbf{z}+\check{\boldsymbol{\vartheta}})\check{\boldsymbol{\vartheta}}^{\mathbf{k}}
=\sum_{\mathbf{k}\in\mathbb{Z}_+^2}\dfrac{1}{\mathbf{k}!}\frac{\partial^{\mathbf{k}}\gamma_j}{\partial \mathbf{z}^\mathbf{k}}(\mathbf{z}+\boldsymbol{\vartheta})\check{\boldsymbol{\vartheta}}^{\mathbf{k}}.
\end{equation*}
Furthermore, combining with the definitions of $\frac{\partial\Gamma}{\partial z_0}$ and $\frac{\partial\Gamma}{\partial z_1}$, we have that for all $\mathbf{z}\in \mathbb{W}(b^*)^2$ and $\boldsymbol{\vartheta}\in \tilde{I}_{\beta}^2$,
\begin{equation*}
    \Gamma(\mathbf{z})-\Gamma(\mathbf{z}+\boldsymbol{\vartheta})=\left(\dfrac{\partial\Gamma}{\partial z_0}(\mathbf{z}+\boldsymbol{\vartheta}), \dfrac{\partial\Gamma}{\partial z_1}(\mathbf{z}+\boldsymbol{\vartheta})\right)\check{\boldsymbol{\vartheta}}^T+\mathbf{E}_2(\mathbf{z},\boldsymbol{\vartheta}),
\end{equation*}
where $\mathbf{E}_2(\mathbf{z},\boldsymbol{\vartheta})$ is defined by
\begin{equation*}
    \mathbf{E}_2(\mathbf{z},\boldsymbol{\vartheta}):=\left[ \sum\limits_{\|\mathbf{k}\|_1\ge2}\dfrac{1}{\mathbf{k}!}\frac{\partial^{\mathbf{k}}\gamma_j}{\partial \mathbf{z}^\mathbf{k}}(\mathbf{z}+\boldsymbol{\vartheta})\check{\boldsymbol{\vartheta}}^{\mathbf{k}}:j\in\mathbb{Z}_3\right]^T.
\end{equation*}
Then, through the property of the cross product, we obtain that for all $\mathbf{z}\in \mathbb{W}(b^*)^2$ and $\boldsymbol{\vartheta}\in \tilde{I}_{\beta}^2$,
\begin{equation}\label{Gamma_taylor}
			\left|(\Gamma(\mathbf{z})-\Gamma(\mathbf{z}+\boldsymbol{\vartheta}))\cdot \left(\dfrac{\partial\Gamma}{\partial z_0}\times\dfrac{\partial\Gamma}{\partial z_1}\right)(\mathbf{z}+\boldsymbol{\vartheta})\right|
		=\left|\mathbf{E}_2(\mathbf{z},\boldsymbol{\vartheta})
		\cdot \left(\dfrac{\partial\Gamma}{\partial z_0}\times\dfrac{\partial\Gamma}{\partial z_1}\right)(\mathbf{z}+\boldsymbol{\vartheta})\right|.
\end{equation}
We next establish the bounds of the terms in the right hand side of \eqref{Gamma_taylor}.

Since $b^*< R_1:= \frac{R_0}{2}$, Cauchy's integral formula ensures that for all $ \mathbf{k}\in\mathbb{Z}_+^2 $, $j\in \mathbb{Z}_3$, $\mathbf{z}\in \mathbb{W}(b^*)^2$ and $\boldsymbol{\vartheta}\in \tilde{I}_{\beta}^2$,
\begin{equation}\label{Cauchy_integral_formula}
	\left|\frac{\partial^{\mathbf{k}}\gamma_j}{\partial \mathbf{z}^\mathbf{k}}(\mathbf{z}+\boldsymbol{\vartheta})\right| = \dfrac{\mathbf{k}!}{(2\pi)^2}\left|\int_{\tau^2_{R_1}(\mathbf{z}+\boldsymbol{\vartheta})}\dfrac{\gamma_j(\mathbf{w})}{(\mathbf{w}-\mathbf{z}-\boldsymbol{\vartheta})^{\mathbf{k}+1}}d\mathbf{w}\right|\leq \dfrac{\mathbf{k}!}{R_1^{\|\mathbf{k}\|_1}}M_{\gamma},
\end{equation}
where $\boldsymbol{\tau}_{R_1}(\mathbf{z}+\boldsymbol{\vartheta}):=\{ \mathbf{w}:=[w_0,w_1]\in\mathbb{C}^2: |w_{\iota}-z_{\iota}-\vartheta_\iota|=R_1 {~\rm for~} \iota\in\mathbb{Z}_2 \}$ and $ M_{\gamma}:=\max\{ | \gamma_\iota(\mathbf{z}) |: \iota\in\mathbb{Z}_3 {~\rm and~} \mathbf{z}\in \boldsymbol{\tau}_{R_1}(\mathbf{z}+\boldsymbol{\vartheta}) \} $.
It is easy to check that for all $\boldsymbol{\vartheta}\in \tilde{I}_\beta^2$ and $ \mathbf{k}\in\mathbb{Z}_+^2 $ with $ \|\mathbf{k}\|_1\ge2 $,
\begin{equation}\label{vector_k_bound}
|\check{\boldsymbol{\vartheta}}^{\mathbf{k}}|\leq \|\check{\boldsymbol{\vartheta}}\|^2_2(\beta R_1)^{\|\mathbf{k}\|_1-2}.
\end{equation}
By combining \eqref{Cauchy_integral_formula} and \eqref{vector_k_bound}, and after some computation, we derive that
\begin{align*}
\sum_{\|\mathbf{k}\|_1\ge 2}\left|\dfrac{1}{\mathbf{k}!}\frac{\partial^{\mathbf{k}}\gamma_j}{\partial \mathbf{z}^\mathbf{k}}(\mathbf{z}+\boldsymbol{\vartheta})\check{\boldsymbol{\vartheta}}^{\mathbf{k}}\right|
\le \dfrac{(3-2\beta)M_{\gamma}}{(1-\beta)^2R_1^2}\|\check{\boldsymbol{\vartheta}}\|^2_2.
\end{align*}
Thus, following the definition of $\mathbf{E}_2(\mathbf{z},\boldsymbol{\vartheta})$, we have that for all $\mathbf{z}\in \mathbb{W}(b^*)^2$ and $\boldsymbol{\vartheta}\in \tilde{I}_\beta^2$,
$$
\left\|\mathbf{E}_2(\mathbf{z},\boldsymbol{\vartheta})\right\|_2\leq  \dfrac{\sqrt{3}(3-2\beta)M_{\gamma}}{(1-\beta)^2R_1^2}\|\check{\boldsymbol{\vartheta}}\|_2^2.
$$
Meanwhile, the continuity of $\Gamma$ ensures that there exists a positive constant $c_1$ such that for all $\mathbf{z}\in \mathbb{W}(b^*)^2$ and $\boldsymbol{\vartheta}\in \tilde{I}_{\beta}^2$,
$$
\left\|\left(\dfrac{\partial\Gamma}{\partial z_0}\times\dfrac{\partial\Gamma}{\partial z_1}\right)(\mathbf{z}+\boldsymbol{\vartheta})\right\|_2\leq c_1.
$$
Consequently, noting that $\|\check{\boldsymbol{\vartheta}}\|_2=\zeta(\boldsymbol{\vartheta})$, from \eqref{Gamma_taylor}, we deduce the existence of a constant $c_2$ such that \eqref{Gamma_ineq} holds for all $\mathbf{z}\in \mathbb{W}(b^*)^2$ and $\boldsymbol{\vartheta}\in \tilde{I}_\beta^2$. Hence, we obtain our desired result \eqref{Gamma_ineq}.
\end{proof}

With the help of Lemmas \ref{tilde_d_lwrbnd_Su} and \ref{lemma_upper_numer_G}, we demonstrate that under assumptions \textbf{(A1)} and \textbf{(A2)}, the function $G$ meets the conditions specified in Theorem \ref{four_decay_general}, as outlined in the following proposition.	
	
\begin{proposition}\label{analy_contn_radius}
If assumptions  \textup{\textbf{(A1)}} and \textup{\textbf{(A2)}} hold, then the periodic extension of $G$ with respect to $\boldsymbol{\theta}$ satisfies the conditions in Theorem \ref{four_decay_general}, where $r$ in the statement of Theorem \ref{four_decay_general} equals to $b^*$.
\end{proposition}
\begin{proof}
Assumption \textup{\textbf{(A2)}} ensures that for each $\boldsymbol{\vartheta}\in I_{2\pi}^2$, the periodic extensions of both functions $(\Gamma(\cdot)-\Gamma(\cdot+\boldsymbol{\vartheta}))\cdot \left(\frac{\partial\Gamma}{\partial z_0}\times\frac{\partial\Gamma}{\partial z_1}\right)(\cdot+\boldsymbol{\vartheta})$ and $d(\cdot,\boldsymbol{\vartheta})$ can be analytically extended to $\mathbb{W}(R_0)^2$. Meanwhile, Lemma \ref{tilde_d_lwrbnd_Su} shows that for all $\boldsymbol{\vartheta}\in I_{2\pi}^2$ with $\zeta(\boldsymbol{\vartheta})\neq 0$, the function $d(\cdot,\boldsymbol{\vartheta})$ has no zeros in  $ \mathbb{W}(b^*)^2$. Hence, it can be easily seen from \eqref{eq_Gd} that the periodic extension of $G$ with respect to $\boldsymbol{\theta}$ can be uniformly analytic extended to  $\mathbb{W}(b^*)^2$ with respect to $\boldsymbol{\theta}$. This means the periodic extension of $G$ with respect to $\boldsymbol{\theta}$ satisfies the conditions of Lemma \ref{bivarite_four_coe}.

On the other hand, integrating \eqref{eq_Gd} with the Lemmas \ref{tilde_d_lwrbnd_Su} and \ref{lemma_upper_numer_G} shows that there exists $c$ such that for all $\mathbf{z}\in \mathbb{W}(b^*)^2$ and $\boldsymbol{\vartheta}\in I^2_{2\pi}$, $|G(\mathbf{z},\boldsymbol{\vartheta})|\leq \frac{c}{{C_1^{3/2}}}(\zeta(\boldsymbol{\vartheta}))^{-1}$. It can be easily seen that the function $(\zeta(\boldsymbol{\vartheta}))^{-1}$, $\boldsymbol{\vartheta}\in I^2_{2\pi}$, is Lebesgue integrable on $I_{2\pi}^2$. Therefore,
from the definition of $S_{\delta,G}(\boldsymbol{\vartheta})$, we know that there is $M>0$ such that
for all $0<\delta<b^*$,
$$
\int_{I_{2\pi}^2} S_{\delta,G}(\boldsymbol{\vartheta})d\boldsymbol{\vartheta}\leq M,
$$
which is the desired condition \eqref{eq_four_decay_general0} in Theorem \ref{four_decay_general}.
\end{proof}

Now we estimate the decay pattern of matrix entries $ K_{\mathbf{k},\mathbf{l}}$ by combining Theorem \ref{four_decay_general} and Proposition \ref{analy_contn_radius}.

\begin{corollary}\label{upperbnd_of_matrix_element}
If assumptions \textup{\textbf{(A1)}} and \textup{\textbf{(A2)}} hold, then there exists a constant $M>0$ such that for all $\mathbf{k}, \mathbf{l}\in\mathbb{Z}^2$,
\begin{equation}\label{upper_bnd_element}
	\left\vert K_{\mathbf{k},\mathbf{l}}\right\vert
	\leq \frac{M}{2\pi} e^{-b\Vert\mathbf{l}-\mathbf{k}\Vert_1}.
\end{equation}
\end{corollary}
\begin{proof}
As $ 0<b<b^* $, Theorem \ref{four_decay_general} and Proposition \ref{analy_contn_radius} imply that there exists a constant $M>0$ such that for all $\mathbf{k}, \mathbf{l}\in\mathbb{Z}^2$,
\begin{equation}\label{G_fourier_uppbnd}
	|\hat{G}_{\mathbf{k},\mathbf{l}}|\leq \frac{M}{2\pi} e^{-b\Vert\mathbf{k}\Vert_1}.
\end{equation}
Meanwhile, the definition of $G$ shows that for all $\mathbf{k}, \mathbf{l}\in\mathbb{Z}^2$,
\begin{eqnarray*}
K_{\mathbf{k},\mathbf{l}}=\left<\mathcal{K}e_{\mathbf{l}}, e_\mathbf{k}\right>&=&\int_{I_{2\pi}^2}\int_{I_{2\pi}^2}K(\boldsymbol{\theta},\boldsymbol{\eta})e_{-\mathbf{k}}(\boldsymbol{\theta})e_{\mathbf{l}}(\boldsymbol{\eta})d \boldsymbol{\theta}d\boldsymbol{\eta}\\
&=&\int_{I_{2\pi}^2}\int_{I_{2\pi}^2}G(\boldsymbol{\theta},\boldsymbol{\vartheta})e_{\mathbf{l}-\mathbf{k}}(\boldsymbol{\theta})e_{\mathbf{l}}(\boldsymbol{\vartheta})d \boldsymbol{\theta}d\boldsymbol{\vartheta}=\hat{G}_{\mathbf{k}-\mathbf{l},-\mathbf{l}}.
\end{eqnarray*}
Combining the above equality with the inequality \eqref{G_fourier_uppbnd} yields the desired result \eqref{upper_bnd_element}.
\end{proof}
	
%\begin{remark}
%We have demonstrated that when the Green's function $K$ is restricted to the tensor product of a surface $\partial \Omega$ that is diffeomorphic to a torus, and the shear transformation $\boldsymbol{\eta}=\boldsymbol{\theta}+\boldsymbol{\vartheta}$ is applied, the resulting function $G(\boldsymbol{\theta}, \boldsymbol{\vartheta})$, $(\boldsymbol{\theta}, \boldsymbol{\vartheta})\in I_{2\pi}^2$, can be uniformly analytically extended with respect to $\boldsymbol{\theta}$. We believe this to be an inherent property of the Green's function itself, suggesting that a similar conclusion may also be applicable to general surfaces.
%\end{remark}

\section{Stability and convergence analysis}
\label{sec:4}

In this section, we consider the stability and convergence of the proposed fast Fourier-Galerkin method.
We will demonstrate that the proposed truncation strategy does not ruin the stability, and achieves a quasi-optimal convergence order. The unique solvability of \eqref{truncated_matrix_equation} results from the stability of the fast Fourier–Galerkin method.
	
To analyze the proposed fast Fourier-Galerkin method, we convert linear system \eqref{truncated_matrix_equation} to an operator equation form.
For all $n\in\mathbb{N}$, we def\mbox{}ine the operator $\widetilde{\mathcal{K}}_n$ by
\begin{equation*}
(\widetilde{\mathcal{K}}_n\varphi)(\boldsymbol{\theta}):=\int_{I_{2\pi}^2}\widetilde{K}_n(\boldsymbol{\theta}, \boldsymbol{\phi})\varphi(\boldsymbol{\phi})
d{\boldsymbol{\phi}},\quad \varphi\in L_2(I^2_{2\pi})~\mathrm{and}~\boldsymbol{\theta}\in I^2_{2\pi},
\end{equation*}
where $\widetilde{K}_n(\boldsymbol{\theta}, \boldsymbol{\phi}):=\sum_{\mathbf{k}\in \mathbb{Z}_{-n,n}^2}\sum_{\mathbf{l}\in \mathbb{Z}_{-n,n}^2}\widetilde{K}_{\mathbf{k},\mathbf{l}}e_{\mathbf{k}}(\boldsymbol{\theta})e_{\mathbf{l}}(\boldsymbol{\phi})$, $\boldsymbol{\theta}\in I^2_{2\pi}~\mathrm{and}~ \boldsymbol{\phi}\in I^2_{2\pi}$.
Then solving the linear system \eqref{truncated_matrix_equation} is equivalent to finding $\tilde{\rho}_n\in X_n$ such that
\begin{equation}\label{trunc_operator_eq}
	\tilde{\rho}_n-\widetilde{\mathcal{K}}_n\tilde{\rho}_n=\mathcal{P}_n g.
\end{equation}
	
We first review the stability and convergence order of the Fourier-Galerkin method for solving \eqref{Galerkin_operator_eq} without truncation strategies. Since we were unable to find discussions in the literature regarding the stability and convergence of the Fourier-Galerkin method for solving \eqref{Galerkin_operator_eq}, we provide a brief discussion on these topics in the following theorem to ensure the self-containment of this paper. This discussion is based on the contents from \cite{atkinson1996numerical,chen2015multiscale,kress2013linear}. To enhance the readability of the article, we recall the following notations.
For $p\in\mathbb{N}$, we denote by $H_p(I_{2\pi}^2)$ the Sobolev space of all functions $\varphi\in L_2(I_{2\pi}^2)$ with the property $\sum_{\mathbf{k}\in\mathbb{Z}^2}(1+\|\mathbf{k}\|_2^2)^p|\hat{\varphi}_\mathbf{k}|^2<\infty$.
For $\varphi\in H_p(I_{2\pi}^2)$, the norm is given by
$
\|\varphi\|_{H_p}:=\left(\sum_{\mathbf{k}\in\mathbb{Z}^2}(1+\|\mathbf{k}\|_2^2)^p|\hat{\varphi}_\mathbf{k}|^2\right)^{1/2}.
$
It is well known that for any $\varphi\in H_p(I_{2\pi}^2)$,
\begin{equation}\label{sobolev_proj_error_upperbnd}
	\Vert\varphi- \mathcal{P}_n\varphi\Vert \leq n^{-p}\Vert\varphi\Vert_{H_p}.
\end{equation}

%From the definitions of $L_2$ norm $\|\cdot\|$ and the Sobolev norm $\|\cdot\|_{H_p}$, we know for any $\varphi\in H_p(I_{2\pi}^2)$,
%\begin{equation}\label{sobolev_proj_error_upperbnd}
%	\Vert\varphi- \mathcal{P}_n\varphi\Vert \leq n^{-p}\left(\sum_{\mathbf{k}\in\mathbb{Z}^2\backslash\mathbb{Z}_{-n,n}^2}(1+\|\mathbf{k}\|_2^2)^p|\hat{\varphi}_\mathbf{k}|^2\right)^{1/2}\leq n^{-p}\Vert\varphi\Vert_{\textcolor{red}{H_p}}.
%\end{equation}
	
\begin{theorem}\label{Galerkin_error}
Let $p\in\mathbb{N}$. If assumptions \textup{\textbf{(A1)}} and \textup{\textbf{(A2)}} hold, then the inverse operators of $\mathcal{I}-\mathcal{K}_n$, $n\in\mathbb{N}$, exist and are uniformly bounded for sufficiently large $n$. Additionally, there exists a positive constant $c$ such that for all $g\in H_p(I_{2\pi}^2)$ and sufficiently large n,
\begin{equation}\label{sol_error_1}
	\| \rho-\rho_n\|\leq cn^{-p}\|\rho\|_{H_p},
\end{equation}
where $g$ appears on the right hand side of \eqref{Galerkin_operator_eq}, and $\rho_n$ is the solution of \eqref{Galerkin_operator_eq}.
\end{theorem}
\begin{proof}
Combining the equality \eqref{eq_Gd} with the Lemmas \ref{tilde_d_lwrbnd_Su} and \ref{lemma_upper_numer_G} shows that there is $c_0>0$ such that for all $\mathbf{z}\in \mathbb{W}(b)^2$ and $\boldsymbol{\vartheta}\in I^2_{2\pi}$,
$|G(\mathbf{z},\boldsymbol{\vartheta})|\leq \frac{c_0}{{C_1^{3/2}}}(\zeta(\boldsymbol{\vartheta}))^{-1}$. Then, from the definition of $G$,  it follows that there exists a positive constant $c_1$ such that for all $\boldsymbol{\theta}, \boldsymbol{\phi}\in I^2_{2\pi}$,
\begin{equation*}\label{kernel_upperbnd}
|K(\boldsymbol{\theta}, \boldsymbol{\phi})|\leq c_1(\zeta(\boldsymbol{\theta}-\boldsymbol{\phi}))^{-1}.
\end{equation*}
This confirms that the operator $\mathcal{K}$ is compact on $L_2(I_{2\pi}^2)$ (see Theorem 2.1.7 in \cite{chen2015multiscale}). Furthermore, the projection operators $\mathcal{P}_n$ converge pointwise to $\mathcal{I}$ on $L_2(I_{2\pi}^2)$, i.e. $\mathcal{P}_n\varphi\rightarrow{\varphi}$ as $n\rightarrow{\infty}$ for every $\varphi\in L_2(I_{2\pi}^2)$. This, along with the compactness of $\mathcal{K}$, implies that $\|\mathcal{K}-\mathcal{K}_n\|\rightarrow{0}$ as $n\rightarrow{\infty}$ (see Lemma 3.1.2 in \cite{atkinson1996numerical}).

Meanwhile, the compactness of $\mathcal{K}$ ensures that the operator $\mathcal{I}-\mathcal{K}$ is injective on $L^{2}(I_{2\pi}^2)$, as proven in Theorem 6.23 in \cite{kress2013linear}.
Consequently, Theorem 3.1.1 in \cite{atkinson1996numerical} guarantees that there is $n^*>0$ such that for all $n\in\mathbb{N}$ with $n>n^*$, the operator $(\mathcal{I}-\mathcal{K}_n)^{-1}$ exists as a bounded operator on $L_2(I_{2\pi}^2)$, and it holds that $$\sup_{n\geq n^*}\|(\mathcal{I}-\mathcal{K}_n)^{-1}\| < \infty.$$ Furthermore, for all $n\geq n^*$, it holds that $$\|\rho-\rho_n\|\leq \| (\mathcal{I}-\mathcal{K}_n)^{-1}\| \| \rho-\mathcal{P}_n\rho\|.$$
Additionally, since $g$ is in $H_p(I_{2\pi}^2)$, $\rho$ is also in $H_p(I_{2\pi}^2)$ as detailed in section 9.1.4 in \cite{atkinson1996numerical}. Therefore, the boundedness of $\sup_{n\geq n^*}\|(\mathcal{I}-\mathcal{K}_n)^{-1}\|$ combined with the inequality  \eqref{sobolev_proj_error_upperbnd} leads to the result \eqref{sol_error_1}.
\end{proof}

Next, we establish the stability and convergence order of the proposed fast Fourier-Galerkin method by estimating the difference between the matrices $\mathbf{K}_N$ and $\widetilde{\mathbf{K}}_N$.
As preparation, we introduce the following notations. For sets $A$ and $B$, $A\setminus B$ represents the set $ \{ x:x\in A \mathrm{~ and ~} x\notin B \}$.
For any $q>0$, $n\in \mathbb{N}$ and $[k_0,l_0]\in \mathbb{Z}_{-n,n}^2$, we define $n_0:=\lfloor q\ln N \rfloor$, where $\lfloor a\rfloor$ denotes the greatest integer less than or equal to $a\in\mathbb{R}$.
Recall that $N=(2n+1)^2$.
For $q>0$, $n\in\mathbb{N}$ with $n> q\ln N$, and $\mathbf{l}:=[l_0,l_1]\in\mathbb{Z}_{-n,n}^2$, we define
\begin{equation*}
\mathbb{I}_{\mathbf{l},N}(q):=\{\mathbf{k}\in\mathbb{Z}_{-n,n}^2:|k_0-l_0|>q\ln N\},
\end{equation*}
and
\begin{equation*}
\mathbb{J}_{\mathbf{l},N}(q):=\{\mathbf{k}\in\mathbb{Z}_{-n,n}^2:|k_0-l_0|\leq q\ln N~\mathrm{and}~|k_1-l_1|>q\ln N-|k_0-l_0|\}.
\end{equation*}
The condition $n> q\ln N$ ensures $\mathbb{I}_{\mathbf{l},N}(q)$ and $\mathbb{J}_{\mathbf{l},N}(q)$ are not empty.	
To simplify, we denote $\mathbb{I}_{\mathbf{l},N}(q)$ and $\mathbb{J}_{\mathbf{l},N}(q)$ as $\mathbb{I}_{\mathbf{l},N}$ and $\mathbb{J}_{\mathbf{l},N}$, respectively.
Define $\mathbf{E}_N:=\mathbf{K}_N-\widetilde{\mathbf{K}}_N$.
For any matrix $\mathbf{B}:=[b_{kl}]\in\mathbb{C}^{m\times m}$ with $m\in\mathbb{N}$, let	$\|\mathbf{B}\|_1:=\max_{l\in\mathbb{Z}_m}\left\{\sum_{k\in\mathbb{Z}_m}|b_{kl}|\right\}$, and $\|\mathbf{B}\|_\infty:=\max_{k\in\mathbb{Z}_m}\left\{\sum_{l\in\mathbb{Z}_m}|b_{kl}|\right\}$. From Corollary \ref{upperbnd_of_matrix_element}, we can derive the upper bounds for $\|\mathbf{E}_N\|_1$, $\|\mathbf{E}_N\|_\infty$, and $\|\mathbf{E}_N\|_2$, which are detailed in the following lemma.

\begin{lemma}\label{all_norm0}
Let $q>0$. If assumptions  \textup{\textbf{(A1)}} and \textup{\textbf{(A2)}} hold, then there exists positive constant $c$ such that for all $n\in\mathbb{N}$ with $n> q\ln N$, there hold $\|\mathbf{E}_N\|_1\leq c  N^{-qb}\ln N$, $\|\mathbf{E}_N\|_\infty\leq cN^{-qb}\ln N$ and $\|\mathbf{E}_N\|_2\leq c  N^{-qb}\ln N$.
\end{lemma}
\begin{proof}
We begin by considering the upper bound of $\|\mathbf{E}_N\|_1$.
The definitions of $\mathbb{I}_{\mathbf{l},N}$ and $\mathbb{J}_{\mathbf{l},N}$ show that for all $n\in \mathbb{N}$,
$$\{[\mathbf{k}, \mathbf{l}]: \mathbf{l}\in \mathbb{Z}_{-n,n}^2, \mathbf{k}\in \mathbb{I}_{\mathbf{l},N}\cup \mathbb{J}_{\mathbf{l},N}\}=\mathbb{Z}_{-n,n}^4\setminus\mathbb{L}_N(q).$$ Thus, it follows from
the definitions of $\|\cdot\|_1$ and $\widetilde{\mathbf{K}}_N$ that
\begin{equation}\label{eq_IIJJBref}
	\left\|\mathbf{E}_N\right\|_1	=\max_{\mathbf{l}\in\mathbb{Z}_{-n,n}^2}\left\{ \sum_{\mathbf{k}\in\mathbb{Z}_{-n,n}^2}|{K}_{\mathbf{k},\mathbf{l}}-\widetilde{K}_{\mathbf{k},\mathbf{l}}|\right\}
	\leq \max_{\mathbf{l}\in\mathbb{Z}_{-n,n}^2}\left\{\sum_{\mathbf{k}\in \mathbb{I}_{\mathbf{l},N}}|K_{\mathbf{k},\mathbf{l}}|+\sum_{\mathbf{k}\in \mathbb{J}_{\mathbf{l},N}}|K_{\mathbf{k},\mathbf{l}}|\right\}.
\end{equation}
Corollary \ref{upperbnd_of_matrix_element} provides that there exists a constant $M>0$ such that for $\mathbf{k}, \mathbf{l}\in\mathbb{Z}^2$,
\begin{equation}\label{upperbnd_entriesBref}
	\left|K_{\mathbf{k},\mathbf{l}}\right|\leq \frac{M}{2\pi} e^{-b\|\mathbf{l}-\mathbf{k}\|_1}.
\end{equation}
Substituting \eqref{upperbnd_entriesBref} into \eqref{eq_IIJJBref} leads to
\begin{equation}\label{eq_IIJJ2Bref}
	\left\|\mathbf{E}_N\right\|_1	
	\leq \frac{M}{2\pi} \max_{\mathbf{l}\in\mathbb{Z}_{-n,n}^2}\left\{\sum_{\mathbf{k}\in \mathbb{I}_{\mathbf{l},N}}e^{-b\|\mathbf{l}-\mathbf{k}\|_1}+\sum_{\mathbf{k}\in \mathbb{J}_{\mathbf{l},N}}e^{-b\|\mathbf{l}-\mathbf{k}\|_1}\right\}.
\end{equation}
Therefore, to estimate $\left\|\mathbf{E}_N\right\|_1$,  it is crucial to establish the upper bounds for the sums on the right-hand side of \eqref{eq_IIJJ2Bref}.

The definitions of $\mathbb{I}_{\mathbf{l},N}$ and $n_0$ show that for all $n\in \mathbb{N}$ with $n> q\ln N$, $\mathbf{l}:=[l_0, l_1] \in\mathbb{Z}^2_{-n,n}$, and $\mathbf{k}:=[k_0,k_1]\in \mathbb{I}_{\mathbf{l},N}$, there are $|l_0-k_0|>q\ln N\geq n_0$ and $k_1\in\mathbb{Z}_{-n,n}$. Therefor, for all $n\in \mathbb{N}$ with $n> q\ln N$ and $\mathbf{l} \in\mathbb{Z}^2_{-n,n}$,
\begin{eqnarray}\label{lemma10_1_termBref}
\nonumber	\sum_{ \mathbf{k} \in \mathbb{I}_{\mathbf{l},N}}e^{-b\|\mathbf{l}-\mathbf{k}\|_1}&\leq & \left(\sum_{|l_0-k_0|\geq n_0+1}e^{-b|l_0-k_0|}\right)\left(\sum_{k_1\in \mathbb{Z}}e^{-b|l_1-k_1|}\right)\\
&\leq& \frac{4 e^{-b(n_0+1)}}{(1-e^{-b})^2}.
\end{eqnarray}	
Similarly, from the definitions of $n_0$, $n_1$ and $\mathbb{J}_{\mathbf{l},N}$,  show that for all $n\in \mathbb{N}$ with $n> q\ln N$,  $\mathbf{l} \in\mathbb{Z}^2_{-n,n}$, and $\mathbf{k}\in \mathbb{I}_{\mathbf{l},N}$, there are $|l_0-k_0|\leq q\ln N\leq n_0$ and $|l_1-k_1|\geq n_0-|l_0-k_0|+1$. Hence, for all $n\in \mathbb{N}$ with $n> q\ln N$ and $\mathbf{l} \in\mathbb{Z}^2_{-n,n}$,
\begin{eqnarray}\label{lemma10_1_termBref2}
\nonumber	\sum_{ \mathbf{k} \in \mathbb{J}_{\mathbf{l},N}}e^{-b\|\mathbf{l}-\mathbf{k}\|_1}&\leq& \sum_{|l_0-k_0|\leq n_0}e^{-b|l_0-k_0|}\left(\sum_{|l_1-k_1|\geq n_0-|l_0-k_0|+1}e^{-b|l_1-k_1|}\right)\\
&\leq& 2\sum_{|l_0-k_0|\leq n_0}\frac{e^{-b(n_0+1)}}{1-e^{-b}}.
\end{eqnarray}	
Notably, the cardinality of set $\{k_0\in\mathbb{Z}:|l_0-k_0|\leq n_0\}$ equals $ 2n_0+1 $.
Thus, it follows from \eqref{lemma10_1_termBref2} that for all $n\in \mathbb{N}$ with  $n> q\ln N$, and $\mathbf{l}\in\mathbb{Z}_{-n,n}^2 $,
\begin{equation}\label{app_ineq1Bref}
	\sum_{ \mathbf{k} \in \mathbb{J}_{\mathbf{l},N}}e^{-b\|\mathbf{l}-\mathbf{k}\|_1}\leq \frac{2(2n_0+1 )}{(1-e^{-b})}e^{-b(n_0+1)}.
\end{equation}
Note that $n_0\leq q\ln N$. Substituting \eqref{lemma10_1_termBref} and \eqref{app_ineq1Bref} into \eqref{eq_IIJJ2Bref} yields that there is a positive constant $c_2$ such that for all $n\in\mathbb{N}$ satisfying $n> q\ln N$, $\|\mathbf{E}_N\|_1\leq c_2 N^{-qb}\ln N$.

The definitions of $\|\cdot\|_\infty$ and $\widetilde{K}_N$ ensure that
\begin{equation*}
	\left\Vert\mathbf{E}_N\right\Vert_\infty
	=\max_{\mathbf{k}\in\mathbb{Z}_{-n,n}^2}\left\{ \sum_{\mathbf{l}\in\mathbb{Z}_{-n,n}^2}|{K}_{\mathbf{k},\mathbf{l}}-\widetilde{K}_{\mathbf{k},\mathbf{l}}|\right\}
	\leq\max_{\mathbf{k}\in\mathbb{Z}_{-n,n}^2}\left\{
	\sum_{\mathbf{l}\in\mathbb{I}_{\mathbf{k},N}}|K_{\mathbf{k},\mathbf{l}}|+
	\sum_{\mathbf{l}\in\mathbb{J}_{\mathbf{k},N}}|K_{\mathbf{k},\mathbf{l}}|
	\right\}.
\end{equation*}
Then, applying  the inequality \eqref{upperbnd_entriesBref} and following a similar process as for establishing the upper bound of $\|\mathbf{E}_N\|_1$, we derive that there exists a positive constant $c_3$ such that for all $n\in\mathbb{N}$ with $n> q\ln N$, $\|\mathbf{E}_N\|_\infty\leq c_3 N^{-qb}\ln N$.
Considering $\|\mathbf{E}_N\|_2^2 \leq \|\mathbf{E}_N\|_1 \|\mathbf{E}_N\|\infty$, it follows that $n\in\mathbb{N}$ with $n> q\ln N$, $\|\mathbf{E}_N\|_2 \leq \sqrt{c_2c_3} N^{-qb}\ln N$. By setting $c:=\max\{c_2, c_3\}$, we achieve the desired result.
\end{proof}

With the help of Lemma \ref{all_norm0}, we have the following auxiliary lemma.

\begin{lemma}\label{truncated_operator_norm_error}
Let $q>0$. If assumptions \textup{\textbf{(A1)}} and \textup{\textbf{(A2)}} hold, then there exists a positive constant $c$ such that for all $\varphi\in L_2(I_{2\pi}^2)$ and $n\in\mathbb{N}$ with $n> q\ln N$,
$$
	\left\Vert(\mathcal{K}_n-\widetilde{\mathcal{K}}_n)\mathcal{P}_n\varphi\right\Vert\leq c \|\varphi\| N^{-qb}\ln N  .
$$
\end{lemma}
\begin{proof}
For all $n\in\mathbb{N}$, and $\varphi\in  L_2(I^2_{2\pi})$, define $\boldsymbol{\varphi}_n$ as a row vector $[\hat{\varphi}_{\mathbf{k}}:\mathbf{k}\in\mathbb{Z}_{-n,n}^2]$. From the definition of $\mathbf{E}_N$, it can be verified that for all $n\in\mathbb{N}$, and $\varphi, \psi\in  L_2(I^2_{2\pi})$,
\begin{eqnarray*}
	\left <(\mathcal{K}_n-\widetilde{\mathcal{K}}_n)\mathcal{P}_n\varphi,\psi\right>&=&\left <(\mathcal{K}_n-\widetilde{\mathcal{K}}_n)\mathcal{P}_n\varphi,\mathcal{P}_n\psi\right>\\
	&=&\sum_{ \mathbf{k} \in\mathbb{Z}_{-n,n}^2}\sum_{ \mathbf{l} \in\mathbb{Z}_{-n,n}^2}\hat{\psi}_{\mathbf{k}}(K_{\mathbf{k},\mathbf{l}}-\widetilde{K}_{\mathbf{k},\mathbf{l}})\hat{\varphi}_{\mathbf{l}}\\
	&=& \boldsymbol{\psi}_n \mathbf{E}_N\boldsymbol{\varphi}_n^T.
\end{eqnarray*}
Note that for all $\varphi\in L_2(I^2_{2\pi})$, it must hold $\|\boldsymbol{\varphi}_n\|_2\leq \|\varphi\|$. Therefore, from the above equality, it follows that for all $n\in\mathbb{N}$, and $\varphi\in  L_2(I^2_{2\pi})$,
\begin{equation*}
\|(\mathcal{K}_n-\widetilde{\mathcal{K}}_n)\mathcal{P}_n\varphi\|=\sup_{\mathop{\psi\in L_2(I^2_{2\pi})}\limits_{\| \psi\|=1}}\left|\left<(\mathcal{K}_n-\widetilde{\mathcal{K}}_n)\mathcal{P}_n\varphi,\psi\right>\right|
\leq \|\varphi\|  \|\mathbf{E}_N\|_2.
\end{equation*}
This inequality, combining with Lemma \ref{all_norm0}, yields the desired result.
\end{proof}
	
We are now prepared to establish the stability and convergence analysis of the proposed fast Fourier-Galerkin method. The results of this analysis are presented in the following theorem.

\begin{theorem}
Let $p\in\mathbb{N}$ and assumptions \textup{\textbf{(A1)}}-\textup{\textbf{(A2)}} hold. If the operator $\mathcal{I}-\mathcal{K}$ is injective from $L_{2}(I_{2\pi}^2)$ to $L_{2}(I_{2\pi}^2)$, the right hand side $g$ of \eqref{orig_bie} is in $H_p(I_{2\pi}^2)$,  $0 < b < \min\left\{\frac{C_0^2}{\sqrt{2}M_0M_1},\frac{R_0}{2}\right\}$ and $q>0$ with $qb\geq p/2$ , then there exists a positive constant $c$ such that for sufficiently large $n$ and any $\varphi\in H_p(I_{2\pi}^2)$,
\begin{equation}\label{truncated_stability}
\|(\mathcal{I}-\widetilde{\mathcal{K}}_n)\mathcal{P}_n\varphi\|\geq c \|\varphi\|,
\end{equation}
and
\begin{equation}\label{truncated_error}
  \|\rho-\tilde{\rho}_n\|\leq  cN^{-p/2}\ln N \|\rho\|_{H_p},
\end{equation}
where $\widetilde{\mathcal{K}}_n$ and $\tilde{\rho}_n$ depend on $q$.
\end{theorem}
	
\begin{proof}
According to Theorem \ref{Galerkin_error}, the operator $(\mathcal{I}-\mathcal{K}_n)^{-1}$ exists and is bounded for sufficiently large $n$. Thus, by noting that
$$
\| (\mathcal{I}-\widetilde{\mathcal{K}}_n)\mathcal{P}_n\varphi \|\geq\| (\mathcal{I}-\mathcal{K}_n)\mathcal{P}_n\varphi\|-\|(\mathcal{K}_n-\widetilde{\mathcal{K}}_n)\mathcal{P}_n\varphi\|,
$$
we obtain inequality \eqref{truncated_stability} with Lemma \ref{truncated_operator_norm_error}.

Inequality \eqref{truncated_stability} implies there is a positive constant $c_0$ that for sufficiently large $n$,
\begin{equation}\label{last_thm_eq2}
\|\rho-\tilde{\rho}_n\|\leq \|\rho-\mathcal{P}_n\rho\|+ c_0\|(\mathcal{I}-\widetilde{\mathcal{K}}_n)(\mathcal{P}_n\rho-\tilde{\rho}_n)\|.
\end{equation}
Since $\mathcal{P}_n(\mathcal{I}-\mathcal{K})\rho=(\mathcal{I}-\widetilde{\mathcal{K}}_n)\tilde{\rho}_n$,
we can express
\begin{equation}\label{last_thm_eq3}
(\mathcal{I}-\widetilde{\mathcal{K}}_n)(\mathcal{P}_n\rho-\tilde{\rho}_n)=\mathcal{P}_n(\mathcal{I}-\mathcal{K})(\mathcal{P}_n\rho-\rho)+(\mathcal{K}_n-\widetilde{\mathcal{K}}_n)\mathcal{P}_n\rho.
\end{equation}
Note that for all $q>0$, it holds that $n> q\ln N$ for sufficiently large $n$. Additionally, since $g$ is in $H_p(I_{2\pi}^2)$, $\rho$ is also in $H_p(I_{2\pi}^2)$, as noted in Section 9.1.4 of \cite{atkinson1996numerical}. Thus, by substituting \eqref{last_thm_eq3} into \eqref{last_thm_eq2} along with inequalities \eqref{sobolev_proj_error_upperbnd}, and utilizing Lemma \ref{truncated_operator_norm_error} with the inequality $qb\geq p/2$, we derive our desired inequality \eqref{truncated_error}.
\end{proof}
\section{Numerical experiments}
\label{sec:5}
In this section, we present two numerical examples to confirm the theoretical order of approximation and computational complexity of
the proposed method. The computer program is run on a personal computer running Windows 10, with a 3.00 GHz Intel(R) Core(TM) i5-9500 CPU, and 16 GB RAM.
	
Recall that $\rho_n$ and $\tilde{\rho}_n$ are the solutions to the Fourier-Galerkin and fast Fourier-Galerkin methods, as described in equations \eqref{Galerkin_operator_eq} and \eqref{trunc_operator_eq}, respectively. Due to the lack of an analytic solution for equation \eqref{orig_bie}, we estimate the errors $e_n:=\frac{\Vert\rho-\rho_n\Vert}{\Vert\rho\Vert}$ and $\tilde{e}_n:=\frac{\Vert\rho-\tilde{\rho}_n\Vert}{\Vert\rho\Vert}$  using $\rho_{60}$ as a proxy for $\rho$. In this context, the matrix $\mathbf{K}_N$ scales up to $N=14641$.

The term ``C.O." stands for the convergence order of the methods, defined by $\log_{\chi(n)}\frac{e_n}{e_{n-10}}$ or $\log_{\chi(n)}\frac{\tilde{e}_n}{\tilde{e}_{n-10}}$, where $\chi(n) := \frac{n-10}{n}$. The condition numbers of the matrices $\mathbf{I}_N+\tilde{\mathbf{K}}_N$ and $\mathbf{I}_N+\mathbf{K}_N$ are denoted as ``Cond".	Additionally, the compression ratio ``C.R." of the truncated coefficient matrix $\tilde{\mathbf{K}}_N$ is calculated as $\mathrm{C.R.} := \frac{\mathcal{N}(\tilde{\mathbf{K}}_N)}{\mathcal{N}(\mathbf{K}_N)}$, with $\mathcal{N}(\mathbf{K}_N)$ and $\mathcal{N}(\tilde{\mathbf{K}}_N)$ indicating the number of nonzero entries in $\mathbf{K}_N$ and $\tilde{\mathbf{K}}_N$, respectively.

\begin{figure}[h]
\centering
	{\includegraphics[height=4.4cm,width=5.2cm]{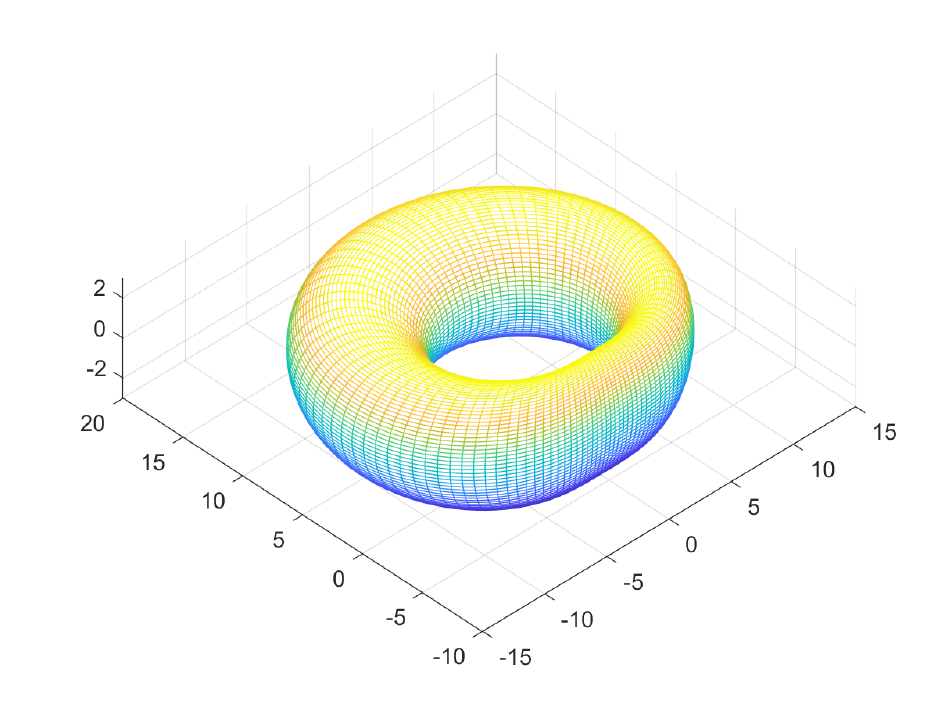}}
		%  diffmeanineq.png
		% \caption{figure 1}
		%\label{fig:revol3}}
	\hspace{0.4in}
	{\includegraphics[height=4.4cm,width=5.2cm]{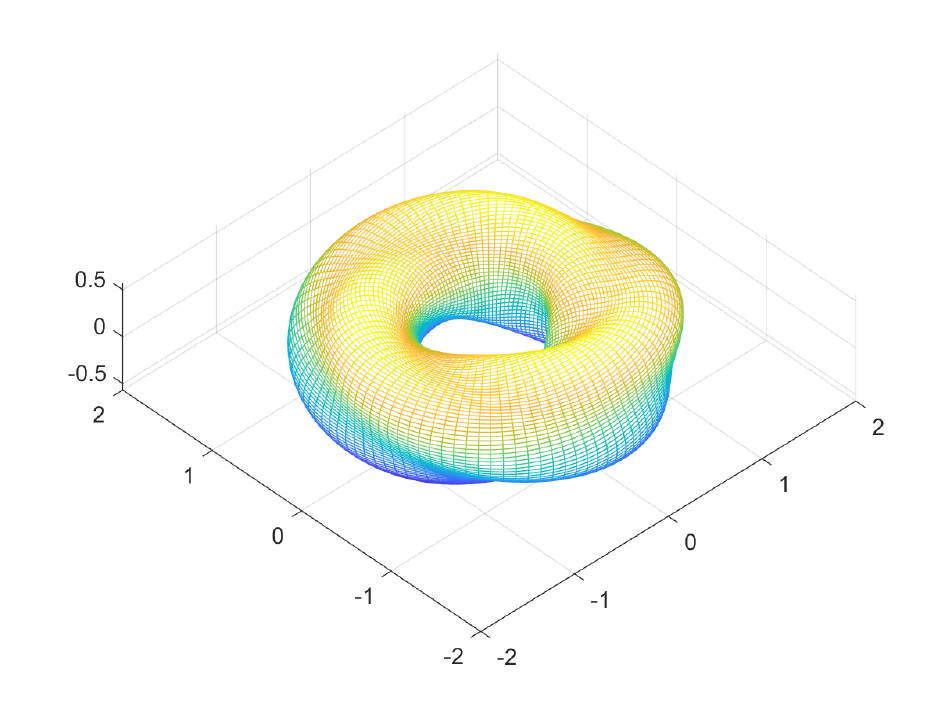}}
	\caption{The bagel-shaped surface (left) and the cruller surface (right)}
	\label{fig: examples}
\end{figure}

\noindent\textbf{Example 1}. We consider solving the boundary integral equation \eqref{bie1} on the bagel-shaped surface $\partial D$, depicted on the left side of Figure \ref{fig: examples} and represented by the parametric equation
\begin{equation*}
\Gamma(\boldsymbol{\theta})=[(5+3\cos\theta_0)(1+0.5\sin\theta_1)\cos\theta_1,(5+3\cos\theta_0)(1+0.5\sin\theta_1)\sin \theta_1,3\sin\theta_0)]^T,
\end{equation*}
where $\boldsymbol{\theta}:=[\theta_0,\theta_1]\in I_{2\pi}^2$.

The right hand side function of \eqref{bie1} is defined as $h(\mathbf{x})=2|\mathbf{x}-\mathbf{x}_0|^2\ln|\mathbf{x}-\mathbf{x}_0|$, where $\mathbf{x}_0= \Gamma(\pi,\pi)\in\partial D$ and $\mathbf{x}\in\partial D$.
Obviously, each component of this parameter equation is an infinity differentiable and $2\pi$-biperiodic. It can be checked that $h\in H_{3-\epsilon}(\partial D)$, where $\epsilon>0$ is
an arbitrary number, implying $g\in H_{3-\epsilon}(I_{2\pi}^2)$. Moreover, based on Section 9.1.4 of \cite{atkinson1996numerical}, $\rho$ belongs to $H_{3-\epsilon}(I_{2\pi}^2)$, ensuring a theoretical convergence order of $3-\epsilon$. In this experiment, we select $\delta=1$ and $q=2.3$.
	
We present in Table \ref{example 1} the relative errors, convergence orders, and condition numbers for both the Fourier-Galerkin method (referred to as `non-truncated') and the fast Fourier-Galerkin method  (referred to as `truncated').
The relative errors $e_n$ of the Fourier-Galerkin method and $\tilde{e}_n$ of the fast Fourier-Galerkin method for Example 1 are listed in the second and the fifth columns, respectively. The last column shows the ``C.R." values of the coefficient matrix for the proposed fast Fourier-Galerkin method. It can be seen that the ``C.R." values of the truncated matrix $\tilde{\mathbf{K}}_N$ decreases as $n$ increases. Furthermore, despite the coefficient matrix $\tilde{\mathbf{K}}_N$  exhibits a remarkably high level of sparsity, $\tilde{e}_n$ remains comparable to $e_n$. Additionally, the condition numbers, listed in the fourth and seventh columns, are consistent at 3.9851, affirming that the fast Fourier-Galerkin method is stable as the Fourier-Galerkin method. The column titled ``C.O." verifies our theoretical convergence order.

\begin{table}[htbp]
	\setlength{\parskip}{0.2cm plus4mm minus3mm}
	\caption{Numerical results for Example 1: Solving the boundary integral equation \eqref{bie1} on the bagel-shaped surface}
	\centering
	\begin{tabular}{c|ccc|cccc}
    	\hline
	    \multicolumn{1}{c|}{\multirow{2}{*}{n}} &
		\multicolumn{3}{c|}{non-truncated} & \multicolumn{4}{c}{truncated} \\ \cline{2-8}
			& $e_n$ & C.O. &  Cond  & $\tilde{e}_n$ & C.O. &Cond &  C.R.
		\\ \hline
			15 & 8.8193e-06 & -     & 3.9851 & 1.0644e-05  & -    & 3.9851 & 0.3443\\
			25 & 3.1172e-06 & 2.04  & 3.9851 & 3.2692e-06  & 2.31 & 3.9851 & 0.2026\\
			35 & 1.4445e-06 & 2.29  & 3.9851 & 1.5206e-06  & 2.27 & 3.9851 & 0.1243\\
			45 & 7.0882e-07 & 2.83  & 3.9851 & 7.6028e-07  & 2.76 & 3.9851 & 0.0868\\
		\hline
	\end{tabular}
	\label{example 1}
\end{table}
	
\noindent\textbf{Example 2}. In this example we consider solving the boundary integral equation \eqref{bie1} on the cruller surface $\partial D$, which is shown on the right side of Figure \ref{fig: examples}. The surface is defined by the parametric equation
\begin{equation*}
\Gamma(\boldsymbol{\theta})=\left[(1+\omega(\boldsymbol{\theta})\cos\theta_0)\cos\theta_1,(1+\omega(\boldsymbol{\theta})\cos\theta_0)\sin\theta_1,\omega(\boldsymbol{\theta})\sin\theta_0\right]^T,
\end{equation*}
where $\omega(\boldsymbol{\theta}):=1/2+0.065\cos(3\theta_0+3\theta_1)$,
and $\boldsymbol{\theta}=[\theta_0,\theta_1]\in I^2_{2\pi}$.

It is evident that each component of function $\Gamma$ is infinitely differentiable, and $2\pi$-biperiodic. We define the right hand side function $h(\mathbf{x})=(x_0-\pi)^2\ln|x_0-\pi|$, where $\mathbf{x}_0= \Gamma(\pi,\pi)$ and $\mathbf{x}:=[x_0,x_1,x_2]^T\in\partial D$. Noting that $h\in H_{3-\epsilon}(\partial D)$, with $\epsilon>0$, we infer from Section 9.1.4 of \cite{atkinson1996numerical} that $\rho\in H_{3-\epsilon}(I^2_{2\pi})$.  This establishes a theoretical convergence order of $3 - \epsilon$. For our computations, we set $\delta = 1$ and $q = 2.3$.

In Table \ref{example 2}, we present the relative errors, denoted as $e_n$ and $\tilde{e}_n$, obtained from the Fourier-Galerkin method and the fast Fourier-Galerkin method for Example 2, respectively. These errors are displayed in the second and fifth columns. Additionally, the column labeled ``C.O." lists the convergence order.

As shown in the fourth and seventh columns of Table \ref{example 2}, the condition numbers of the matrices $\mathbf{I}_N - \mathbf{K}_N$ and $\mathbf{I}_N - \tilde{\mathbf{K}}_N$ consistently approach approximately $3.6919$. The last column of Table \ref{example 2} highlights the sparsity of the coefficient matrix $\tilde{\mathbf{K}}_N$ in our proposed method, as indicated by the ``C.R." values.  Furthermore, when $\tilde{\mathbf{K}}_N$ is employed as the coefficient matrix, the relative error of the solutions remains comparable to those obtained using the non-truncated matrix $\mathbf{K}_N$.
	
These two examples illustrate that by employing the proposed fast Fourier-Galerkin method, we significantly increase the sparsity of the coefficient matrix. Additionally, the solutions obtained from the fast Fourier-Galerkin method \eqref{trunc_operator_eq} closely resemble those obtained using the original Fourier-Galerkin method  \eqref{Galerkin_operator_eq} in terms of error and convergence order. Furthermore, the condition numbers of the coefficient matrices remain comparable before and after truncation.

	%It can be seen from the above two examples that by employing the proposed fast Fourier-Galerkin method, we significantly enhance the sparsity of the coefficient matrix. However, it is important to note that the condition number of the coefficient matrix remains unaffected. Moreover, the relative error and convergence order of the solution to equation \eqref{trunc_operator_eq} closely resemble those obtained from solving equation \eqref{Galerkin_operator_eq}.

 \begin{table}[h]
		\setlength{\parskip}{0.2cm plus4mm minus3mm}
		\caption{Numerical results for Example 2: Solving the boundary integral equation \eqref{bie1} on the cruller surface}
		\centering
		\begin{tabular}{c|ccc|cccc}
			\hline
			\multicolumn{1}{c|}{\multirow{2}{*}{$n$}}
			&
			\multicolumn{3}{c|}{non-truncated} & \multicolumn{4}{c}{truncated} \\ \cline{2-8}
			& $e_n$ & C.O. &  Cond  & $\tilde{e}_n$ & C.O. &Cond &  C.R.
			\\ \hline
			15 & 2.1270e-03  & -     & 3.6919  & 2.7481e-03  &  -   & 3.6923 & 0.3443\\
			25 & 7.2865e-04  & 2.10  & 3.6919  & 7.9016e-04  & 2.44 & 3.6920 & 0.2026\\
			35 & 3.4019e-04  & 2.26  & 3.6919  & 4.0777e-04  & 1.97 & 3.6919 & 0.1243\\
			45 & 1.7482e-04  & 2.65  & 3.6919  &  2.3748e-04 & 2.15 & 3.6919 & 0.0868\\
			\hline
		\end{tabular}
		\label{example 2}
	\end{table}
\section{Conclusions}
\label{sec:6}
We introduce a fast Fourier-Galerkin method for solving boundary integral equations arising from the Dirichlet problem for Laplace's equation in domains with non-axisymmetric toroidal surfaces. Our comprehensive analysis of the kernel function helps us understand the decay patterns of its Fourier coefficients, and formulate a high-performance truncation strategy. Through theoretical analysis and numerical experiments, we demonstrate that this truncation strategy maintains stability and does not ruin the convergence order. Looking ahead, we plan to develop a quadrature rule for calculating the entries in the truncated matrix $\tilde{\mathbf{K}}_N$ and to extend our approach to fast algorithms for solving boundary integral equations on other geometric shapes, such as surfaces diffeomorphic to a sphere.

\vspace{1cm}

\noindent{\small\textbf{Funding} This work is partially supported by the National Key Research and Development Program of China (2023YFB3001704),
Key-Area Research and Development Program of Guangdong Province (2021B0101190003), and the Natural Science Foundation of Guangdong
Province, China (2022A1515010831).}\\

\noindent{\small \textbf{Data Availability} The datasets generated during the current study are available from the corresponding author
upon reasonable request.}

\section*{Declarations}
\noindent{\small\textbf{Competing Interests} The authors declare that they have no conflict of interest.}

%\begin{acknowledgements}
%    This research is partially
%supported by the Special Project on High-performance Computing under the National Key R\&D Program (No. 2016YFB0200602), and the Natural Science Foundation of Guangdong Province, China (No.2022A1515010831).
%\end{acknowledgements}




\end{document}